\documentclass[11pt]{amsart}

\usepackage{accents}
\usepackage{appendix}
\usepackage{amsfonts}
\usepackage{amsmath}
\usepackage{amssymb}
\usepackage{amsthm}
\usepackage{array,booktabs,multirow}
\usepackage{cite}
\usepackage{dsfont}
\usepackage[shortlabels]{enumitem}
\usepackage{etoolbox}
\usepackage{float}
\usepackage[hang, flushmargin]{footmisc}
\usepackage{latexsym}
\usepackage{hyperref}
\usepackage{needspace}
\usepackage{tikz}
\usetikzlibrary{matrix,arrows}

\newcommand{\ubar}[1]{\underaccent{\bar}{#1\,}}
\newcommand{\dbar}{\,\mathchar'26\mkern-12mu d \hspace{0.06em}}

\theoremstyle{plain}
\newtheorem{thm}{Theorem}[section]

\newtheorem{lem}[thm]{Lemma}
\newtheorem{prop}[thm]{Proposition}

\theoremstyle{definition}

\theoremstyle{remark}

\AtBeginEnvironment{lem}{\Needspace{3\baselineskip}}%
\AtBeginEnvironment{thm}{\Needspace{3\baselineskip}}
\AtBeginEnvironment{prop}{\Needspace{3\baselineskip}}
\AtBeginEnvironment{cor}{\Needspace{3\baselineskip}}

\setlist[enumerate,1]{leftmargin=2em}

\def\C{\mathbb C}
\def\F{\mathbb F}
\def\H{\mathfrak H}
\def\I{\mathbf I}
\def\N{\mathbb N}
\def\T{\mathbb T}
\def\Z{\mathbb Z}
\def\sl_2{\mathfrak{sl}_2}
\def\U{U_q(\mathfrak{sl}_2)}
\def\V{U_q'(\mathfrak{so}_3)}

\title[Center of the universal Askey--Wilson algebra
at roots of $1$]{Center of
the universal Askey--Wilson algebra\\
at roots of unity}

\author{Hau-Wen Huang}
\address{
Department of Mathematics\\
National Central University\\
Chung-Li 32001 Taiwan
}
\email{hauwenh@math.ncu.edu.tw}

\thanks{The research was partially supported by National Center for Theoretical Sciences of Taiwan and the Council for Higher Education of Israel.
}

\begin{document}

\begin{abstract}
Inspired by a profound observation on the Racah--Wigner coefficients of $U_q(\mathfrak{sl}_2)$, the Askey--Wilson algebras were introduced in the early 1990s. A universal analog $\triangle_q$ of the Askey--Wilson algebras was recently studied.
For $q$ not a root of unity, it is known that $Z(\triangle_q)$ is isomorphic to the polynomial ring of four variables. A presentation for $Z(\triangle_q)$ at $q$ a root of unity is displayed in this paper. %It was shown that the double affine Hecke algebra of type $(C_1^\vee,C_1)$ contains a central extension of Askey--Wilson algebras.
As an application, a presentation for the center of the double affine Hecke algebra of type $(C_1^\vee,C_1)$ at roots of unity is obtained.
\end{abstract}

\maketitle


{\footnotesize {\bf 2010 MSC Primary:} 17B37; {\bf Secondary:} 20C08, 33C45, 33D80.}

\section{Introduction}\label{s:intro}

Throughout this paper an algebra $\mathcal A$ is meant to be an associative algebra with unit and let $Z(\mathcal A)$ denote the center of an algebra $\mathcal A$.

Fix a complex scalar $q\not=0$. In \cite{hidden_sym} Zhedanov proposed the Askey--Wilson algebras which involve five extra parameters $\varrho$, $\varrho^*$, $\eta$, $\eta^*$, $\omega$. Given these scalars the {\it Askey--Wilson algebra} is an algebra over the complex number field $\C$ generated by $K_0$, $K_1$, $K_2$ subject to the relations
\begin{eqnarray*}
qK_1K_2-q^{-1}K_2K_1&=&\omega K_1+\varrho K_0+\eta^*,\\
qK_2K_0-q^{-1}K_0K_2&=&\omega K_0+\varrho^* K_1+\eta,\\
qK_0K_1-q^{-1}K_1K_0&=&K_2.
\end{eqnarray*}
These algebras are named after R. Askey and J. Wilson since the algebras can also describe a hidden relation between the three-term recurrence relation and the $q$-difference equation of Askey--Wilson polynomials \cite{ask85}.
Under the mild assumptions $q^4\not=1$, $\varrho\not=0$ and $\varrho\not=0$ substitute
\begin{gather*}
K_0=-\frac{\sqrt{\varrho^*}A}{q^2-q^{-2}},\qquad \quad
K_1=-\frac{\sqrt{\varrho}B}{q^2-q^{-2}},\qquad \quad
K_2=\frac{\omega}{q-q^{-1}}
-\frac{\sqrt{\varrho\varrho^*}C}{q^2-q^{-2}}
\end{gather*}
into the defining relations of the Askey--Wilson algebra. The resulting relations become that each of
\begin{gather}\label{e:relDelta}
 A+ \frac{qBC-q^{-1}CB}{q^2-q^{-2}},
\qquad \quad
B+
\frac{qCA-q^{-1}AC}{q^2-q^{-2}},
\qquad \quad
C+
\frac{qAB-q^{-1}BA}{q^2-q^{-2}}
\end{gather}
is equal to a scalar. By interpreting the elements in (\ref{e:relDelta}) as central elements, it turns into the so-called {\it universal Askey--Wilson algebra} $\triangle_q$ \cite{uaw2011}. Let us denote $\triangle=\triangle_q$ for brevity.

Let $\alpha$, $\beta$, $\gamma$ denote the central elements of $\triangle$ obtained from multiplying the elements (\ref{e:relDelta}) by $q+q^{-1}$, respectively. Motivated by Zhedanov \cite[\S1]{hidden_sym}, the distinguished central element
\begin{gather}\label{e:Casimir}
q ABC+ q^2 A^2+q^{-2} B^2+q^2 C^2-q A\alpha-q^{-1} B\beta-q C\gamma
\end{gather}
is called the {\it Casimir element} of $\triangle$. For $q$ not a root of unity, the center of $\triangle$ has been shown  in \cite[Theorem~8.2]{uaw2011} to be the four-variable polynomial ring over $\C$ generated by $\alpha$, $\beta$, $\gamma$ and the Casimir element (\ref{e:Casimir}). The inspiration of our study on  $Z(\triangle)$ at roots of unity comes from the quantum group $U_q'(\mathfrak{so}_3)$. The quantum group $U_q'(\mathfrak{so}_n)$ \cite{GK1991} is not Drinfeld--Jimbo type but plays the important roles in the study of $q$-Laplace operators and $q$-harmonic polynomials \cite{ik2001,nuw1996}, $q$-ultraspherical polynomials \cite{ultrapoly:96}, quantum homogeneous spaces \cite{n1996}, nuclear spectroscopy \cite{hkp1999}, $(2+1)$-dimensional quantum gravity \cite{nrz1990,nr1993} and so on. For $n=3$ the quantum group is exactly the Askey--Wilson algebra with $q^4\not=1$, $\varrho=1$, $\varrho^*=1$, $\eta=0$, $\eta^*=0$, $\omega=0$. According to \cite[\S4]{nuw1996} the Casimir element of $U_q'(\mathfrak{so}_3)$ is defined to be
\begin{gather}\label{e:CasimirV}
q(q^2-q^{-2}) K_0 K_1 K_2
-q^2K_0^2-q^{-2}K_1^2-q^2 K_2^2.
\end{gather}

As far as we know, Odesskii \cite[Theorem~4]{odes1986} first found three additional central elements of $\V$ at roots of unity defined as follows.
Assume that $q$ is a primitive $d^{\rm\, th}$ root of unity and set
\begin{gather*}
\dbar=\left\{
\begin{array}{ll}
d \qquad \quad &\hbox{if $d$ is odd},\\
d/2 \qquad \quad &\hbox{if $d$ is even}.
\end{array}
\right.
\end{gather*}
Denote by $\Z$ the ring of integers and by $\N$ the set of the nonnegative integers.
For each $n\in \N$ define
\begin{gather}
T_n(X)
=\sum_{i=0}^{\lfloor n/2\rfloor}(-1)^i \left(
{n-i\choose i}+ {n-i-1\choose i-1}
\right)X^{n-2i}.
\label{e:Tn}
\end{gather}
Here ${n\choose -1}$ for $n\in \N$ and ${-1\choose -1}$ are interpreted as $0$ and $1$, respectively.
Note that $
\frac{1}{2}\,T_n(2X)
$ is the Chebyshev polynomial of the first kind. Then
\begin{gather*}
\Gamma_i=T_{\dbar}(-(q^2-q^{-2})K_i)
\qquad \quad
\hbox{for all $i\in \Z/3\Z$}
\end{gather*}
are central in $\V$. A proof can be found in \cite[Lemma~2]{hp2001}.

On the other hand, while studying the quantum Teichm\"{u}ller space, Checkhov and Fock \cite[Example~2]{fock2000} were aware of a homomorphism $\flat$ of $\V$ into the algebra generated by $x_0^{\pm 1}$, $x_1^{\pm 1}$, $x_2^{\pm 1}$ subject to the relations
\begin{gather*}
x_i x_i^{-1}=x_i^{-1} x_i=1, \qquad \quad
x_i x_{i+1}=q^2 x_{i+1} x_i \qquad \quad \hbox{for all $i\in \Z/3\Z$}.
\end{gather*}
The homomorphism $\flat$ sends
$K_i$ to
\begin{gather*}
\frac{q^{-1} x_i^{-1} x_{i+1}^{-1}+
q x_i^{-1} x_{i+1}+
q^{-1} x_ix_{i+1}}
{q^2-q^{-2}}
\qquad \quad \hbox{for all $i\in \Z/3\Z$}.
\end{gather*}
Presently, the map $\flat$ was shown to be injective by Iorgov \cite[Proposition~1]{iog02}. The images of the Casimir element and $\Gamma_0$, $\Gamma_1$, $\Gamma_2$ were calculated out.
As a consequence the relation
\begin{equation}\label{e:presentationV}
\begin{split}
&2
\left(
\left\lceil
\dbar/2
\right\rceil
-
\left\lfloor
\dbar/2
\right\rfloor
-1
\right)
\left(
\Gamma_0+\Gamma_1+\Gamma_2
+4
\right)
(
(-1)^{\lfloor \dbar/2\rfloor}
T_{\lfloor \dbar/2\rfloor}(\Pi)+2
)
-T_{\dbar}(\Pi)\\
&\quad =\;\;
q^{\dbar}
\Gamma_0\Gamma_1\Gamma_2
+
\Gamma_0^2+\Gamma_1^2+\Gamma_2^2
+
4
\left(
\left\lceil
\dbar/2
\right\rceil
-\left\lfloor
\dbar/2
\right\rfloor
\right)-6
\end{split}
\end{equation}
was first discovered in \cite[Proposition~2]{iog02}, where $\Pi$ is a normalization of (\ref{e:CasimirV}) obtained by multiplying $(q^2-q^{-2})^2$ followed by adding $q^2+q^{-2}$. The center of $\V$ was conjectured by Iorgov in \cite[Conjecture~1]{iog02} to be the commutative algebra over $\C$ generated by $\Gamma_0$, $\Gamma_1$, $\Gamma_2$, $\Pi$ subject to the relation (\ref{e:presentationV}). Until recently, the $\N$-filtration structure of $\V$ was utilized to confirm that $\Gamma_i$ for all $i\in \Z/3\Z$ and $\Pi$ generate $Z(\V)$ by Havl\'i\v{c}ek and Po\v{s}ta \cite[Theorem~3.1]{hp2011}.

As we shall see in this paper, the whole ideas can be generalized to the case of $\triangle$ by a replacement of the role of $\flat$ with an embedding of $\triangle$ into the tensor product of $\U$ with the Laurent ring $\F[a^{\pm 1},b^{\pm 1},c^{\pm 1}]$ given by Terwilliger \cite{uaw&equit2011}. Furthermore, a presentation for $Z(\triangle)$ is given as follows. For each $n\in \N$ let
\begin{eqnarray*}
\phi_n(X_0,X_1,X_2;X) &=& T_n(X) T_n(X_0)+T_n(X_1) T_n(X_2),\\
\psi_n(X_0,X_1,X_2;X) &=& T_{2n}(X)+T_n(X_0)^2+T_n(X_1)^2+T_n(X_2)^2+
T_n(X)T_n(X_0)T_n(X_1)T_n(X_2).
\end{eqnarray*}
To each $i\in \Z/3\Z$ we associate a $\C[X]$-algebra automorphism $\ubar i$ of $\C[X_0,X_1,X_2,X]$ with
\begin{gather*}
X^{\ubar i}_j= X_{i+j}
\qquad \quad
\hbox{for all $j\in \Z/3\Z$}.
\end{gather*}
For each $n\in \N$ there exist unique polynomials $\Phi_n(X_0,X_1,X_2;X)$, $\Psi_n(X_0,X_1,X_2;X)$
over $\Z$ such that
\begin{gather*}
\Phi_n
(
\phi_m^{\ubar 0},
\phi_m^{\ubar 1},
\phi_m^{\ubar 2};
\psi_m)
=
\phi_{mn},
\qquad \quad
\Psi_n(
\phi_m^{\ubar 0},
\phi_m^{\ubar 1},
\phi_m^{\ubar 2};
\psi_m)
=
\psi_{mn}
\qquad \quad
\hbox{for all $m\in \N$}.
\end{gather*}
For instance
\begin{eqnarray*}
\Phi_0(X_0,X_1,X_2;X)&=&8,\\
\Phi_1(X_0,X_1,X_2;X)&=&X_0,\\
\Phi_2(X_0,X_1,X_2;X)&=&
X_0^2-2X+4,\\
\Phi_3(X_0,X_1,X_2;X)&=&
X_0^3-3(X-1)X_0+3X_1X_2
\end{eqnarray*}
and
\begin{eqnarray*}
\Psi_0(X_0,X_1,X_2;X)&=&30,\\
\Psi_1(X_0,X_1,X_2;X)&=&X,\\
\Psi_2(X_0,X_1,X_2;X)&=&
X^2-8X
+2
(
X_0^2+X_1^2+X_2^2
-X_0X_1X_2
)
+10,\\
\Psi_3(X_0,X_1,X_2;X)&=&
X^3-3(
X_0X_1X_2
+
X_0^2+X_1^2+X_2^2
+1
)X
\\
&&\quad +\;\;
3
(
X_0^2X_1^2+X_1^2X_2^2+X_2^2X_0^2
-X_0X_1X_2
)
+6
(
X_0^2+X_1^2+X_2^2
).
\end{eqnarray*}
Let $\Omega$ denote a normalization of (\ref{e:Casimir}) obtained by subtracting from $q^2+q^{-2}$.
Applying the techniques from algebraic number theory, $Z(\triangle)$ is shown to be the commutative algebra over $\C$ generated by $T_{\dbar}(A)$, $T_{\dbar}(B)$, $T_{\dbar}(C)$, $\alpha$, $\beta$, $\gamma$, $\Omega$ subject to the relation
\begin{align*}
&
q^{\dbar} \Phi_{\dbar}^{\ubar 0}(\alpha,\beta,\gamma;\Omega) T_{\dbar}(A)
+q^{\dbar} \Phi_{\dbar}^{\ubar 1}(\alpha,\beta,\gamma;\Omega) T_{\dbar}(B)
+q^{\dbar} \Phi_{\dbar}^{\ubar 2}(\alpha,\beta,\gamma;\Omega) T_{\dbar}(C)
\\
&\quad=\;\;
q^{\dbar} T_{\dbar}(A) T_{\dbar}(B) T_{\dbar}(C)
+
T_{\dbar}(A)^2
+ T_{\dbar}(B)^2
+ T_{\dbar}(C)^2
+\Psi_{\dbar}(\alpha,\beta,\gamma;\Omega)
-2.
\end{align*}

In one year after the Askey--Wilson algebra was formally proposed, the double affine Hecke algebra (DAHA) associated with a reduced affine root system was introduced by Cherednik as a tool to prove several conjectures made by Macdonald \cite{DAHA1992,DAHA1995,DAHA1995v2}. Later, the DAHA was extended by Sahi \cite[\S3]{DAHA1999} to any nonreduced affine root system. Consider the most general DAHA $\H$ of rank $1$, namely of type $(C_1^\vee,C_1)$. Simply speaking, the algebra $\H$ is isomorphic to an algebra over $\C$ generated by $t_0$, $t_1$, $t_0^\vee$, $t_1^\vee$ subject to the relations
\begin{gather*}
(t_0-k_0)(t_0-k_0^{-1})=0,
\qquad \quad
(t_1-k_1)(t_1-k_1^{-1})=0,
\\
(t_0^\vee-k_0^\vee)(t_0^\vee-k_0^{\vee-1})=0,
\qquad \quad
(t_1^\vee-k_1^\vee)(t_1^\vee-k_1^{\vee-1})=0,
\\
t_0^\vee t_0 t_1^\vee t_1=q^{-1},
\end{gather*}
where $k_0$, $k_1$, $k_0^\vee$, $k_1^\vee$ are arbitrary nonzero parameters. The relationships between the Askey--Wilson algebra and the DAHA $\H$ of type $(C_1^\vee,C_1)$ were investigated by Ito, Koornwinder and Terwilliger \cite{it10, koo07, koo08, aw&daha2013}. The results from \cite[\S6]{koo07} imply that there exists a homomorphism $\triangle\to \H$ that sends
\begin{gather*}
A \;\;\mapsto \;\; t_1 t_0^\vee +(t_1t_0^\vee)^{-1},
\qquad  \quad
B \;\;\mapsto \;\; t_1^\vee t_1+(t_1^\vee t_1)^{-1},
\qquad  \quad
C \;\;\mapsto \;\; t_0 t_1+(t_0 t_1)^{-1}.
\end{gather*}

For $q=1$, a presentation for $Z(\H)$ was given by Oblomkov \cite[Theorem~3.1]{ob04}. For $q$ a root of unity, a generalized but indefinite presentation for $Z(\H)$ was subsequently obtained in \cite[Theorem~6.12]{etingof07}.
In the final section, the study of $Z(\triangle)$ is used to explicitly extend the Oblomkov presentation for $Z(\H)$ at roots of unity as below. Let $\phi^{\ubar 0}$, $\phi^{\ubar 1}$, $\phi^{\ubar 2}$, $\psi\in \C$ given by
\begin{eqnarray*}
\phi^{\ubar 0}&=&
(k_1^{\dbar}+k_1^{-\dbar})(k_0^{\vee\dbar}+k_0^{\vee-\dbar})
+q^{\dbar}(k_1^{\vee\dbar}+k_1^{\vee-\dbar})(k_0^{\dbar}+k_0^{-\dbar}),\\
\phi^{\ubar 1}&=&
(k_1^{\dbar}+k_1^{-\dbar})(k_1^{\vee\dbar}+k_1^{\vee-\dbar})
+q^{\dbar}(k_0^{\dbar}+k_0^{-\dbar})(k_0^{\vee\dbar}+k_0^{\vee-\dbar}),\\
\phi^{\ubar 2} &=&
(k_1^{\dbar}+k_1^{-\dbar})(k_0^{\dbar}+k_0^{-\dbar})
+q^{\dbar}(k_0^{\vee\dbar}+k_0^{\vee-\dbar})(k_1^{\vee\dbar}+k_1^{\vee-\dbar}),\\
\psi &=&
k_0^{2\dbar}
+k_1^{2\dbar}
+k_0^{\vee2\dbar}
+k_1^{\vee2\dbar}
+k_0^{-2\dbar}
+k_1^{-2\dbar}
+k_0^{\vee-2\dbar}
+k_1^{\vee-2\dbar}
\\
&&\quad+\;\;
(k_0^{\dbar}+k_0^{-\dbar})
(k_1^{\dbar}+k_1^{-\dbar})
(k_0^{\vee\dbar}+k_0^{\vee-\dbar})
(k_1^{\vee\dbar}+k_1^{\vee-\dbar}).
\end{eqnarray*}
The polynomial ring
$\C[X_0,X_1,X_2]$ modulo the ideal generated by
\begin{gather*}
q^{\dbar} X_0X_1X_2
+X_0^2+X_1^2+X_2^2
-\phi^{\ubar 0} X_0
-\phi^{\ubar 1} X_1
-\phi^{\ubar 2} X_2
+\psi
+4
\end{gather*}
is isomorphic to $Z(\H)$ induced from the mapping
\begin{gather*}
X_0 \;\; \mapsto \;\; (t_1 t_0^\vee)^{\dbar}+(t_1 t_0^\vee)^{-\dbar},
\qquad \quad
X_1 \;\; \mapsto \;\; (t_1^\vee t_1)^{\dbar}+(t_1^\vee t_1)^{-\dbar},
\qquad  \quad
X_2 \;\; \mapsto \;\; (t_0 t_1)^{\dbar}+(t_0 t_1)^{-\dbar}.
\end{gather*}

In addition, this paper will be employed to study the finite-dimensional irreducible $\triangle$-modules at roots of unity in the future paper \cite{Huang:2017}, which has a potential application to the Racah--Wigner coefficients of $\U$ at roots of unity \cite{Huang:2016}. For $q$ not a root of unity, a classification of finite-dimensional irreducible $\triangle$-modules was finished in \cite[Theorem~4.7]{Huang:2015}.

\section{Preliminaries}\label{s:pre}

From now on, the ground field is set to be an arbitrary field $\F$ because the statements in this paper are valid for any ground field not only the complex number field $\C$. The parameter $q$ is always assumed to be a primitive $d^{\rm\, th}$ root of unity. Set $\dbar=d$ if $d$ is odd and $\dbar=d/2$ if $d$ is even.

The purpose of \S\ref{s:uaw} is to give a brief review of the background on $\triangle$ including the Poincar\'{e}--Birkhoff--Witt basis of $\triangle$ and the induced $\N$-filtration structure of $\triangle$. In \S\ref{s:uqsl2} we display the embedding $\triangle\to \F[a^{\pm 1},b^{\pm 1},c^{\pm 1}]\otimes_\F\U$. The Iorgov's identity is displayed in \S\ref{s:qbin} to evaluate $T_{\dbar}(A)^\natural$, $T_{\dbar}(B)^\natural$, $T_{\dbar}(C)^\natural$. Recall the Concini--Kac presentation for $Z(\U)$ in \S\ref{s:preZU} to derive the relation in $Z(\triangle)$ claimed in Introduction.

\subsection{The universal Askey--Wilson algebra $\triangle$}\label{s:uaw}

To drop the assumption $q^4\not=1$, the universal Askey--Wilson algebra $\triangle$ is slightly changed to be an $\F$-algebra with generators $A$, $B$, $C$, $\alpha$, $\beta$, $\gamma$ and the relations assert that each of $\alpha$, $\beta$, $\gamma$ is central in $\triangle$ and
\begin{eqnarray}
CB &=&
q^2 BC+q(q^2-q^{-2})A-q(q-q^{-1})\alpha,\label{e:uaw2}\\
CA &=&
q^{-2} AC-q^{-1}(q^2-q^{-2})B+q^{-1}(q-q^{-1})\beta,\label{e:uaw3}\\
BA &=&
q^2 AB+q(q^2-q^{-2})C-q(q-q^{-1})\gamma.\label{e:uaw1}
\end{eqnarray}
The Poincar\'{e}--Birkhoff--Witt theorem for $\triangle$ was given in \cite[Theorem~4.1]{uaw2011}:

\begin{lem}\label{lem:PBW}
The monomials
\begin{gather*}
A^{i_0} B^{i_1} C^{i_2} \alpha^{j_0} \beta^{j_1} \gamma^{j_2}
\qquad \quad
\hbox{for all $i_0,i_1,i_2,j_0,j_1,j_2\in \N$}
\end{gather*}
form an $\F$-basis of $\triangle$.
\end{lem}

For any two submodules $V$, $W$ of an algebra $\mathcal A$, denote by $V\cdot W$ the submodule of $\mathcal A$ spanned by $vw$ for all $v\in V$ and $w\in W$.
For each $n\in \N$ let $\triangle_n$ denote the $\F$-subspace of $\triangle$ spanned by
\begin{gather*}
A^{i_0} B^{i_1} C^{i_2}  \alpha^{j_0} \beta^{j_1} \gamma^{j_2}
\qquad \quad
\hbox{for all
$i_0,i_1,i_2,j_0,j_1,j_2\in \N$ \; with \;  $i_0+i_1+i_2+j_0+j_1+j_2\leq n$}.
\end{gather*}
Recall from \cite[\S5]{uaw2011} that the $\F$-subspaces $\{\triangle_n\}_{n\in \N}$ of $\triangle$ satisfy
\setlist[enumerate,1]{leftmargin=2.4em}
\begin{enumerate}
\item[(F1)] $\triangle=\bigcup\limits_{n\in \N} \triangle_n$;

\item[(F2)] $\triangle_m\cdot \triangle_n\subseteq \triangle_{m+n}$ for all $m,n\in \N$.
\end{enumerate}
\setlist[enumerate,1]{leftmargin=2em}
In other words, the increasing sequence
\begin{gather*}
\triangle_0\subseteq \triangle_1\subseteq \cdots \subseteq \triangle_n\subseteq \cdots
\end{gather*}
gives an $\N$-filtration of $\triangle$.

Recall from Introduction that the notation $\Omega$ stands for the normalized Casimir element
\begin{gather}\label{e:Omega}
q^2+q^{-2}
-
q ABC-q^2 A^2-q^{-2} B^2-q^2 C^2
+q A\alpha+q^{-1} B\beta+q C\gamma.
\end{gather}
By \cite[Proposition~7.4]{uaw2011} we have

\begin{lem}\label{lem:PBW2}
For each $n\in \N$ the monomials
\begin{gather*}
A^{i_0}B^{i_1}C^{i_2}
\alpha^{j_0} \beta^{j_1} \gamma^{j_2}
\Omega^\ell
\qquad \quad
\begin{split}
&i_0,i_1,i_2,j_0,j_1,j_2,\ell\in \N, \qquad i_0i_1i_2=0,\\
&i_0+i_1+i_2+j_0+j_1+j_2+3\ell\leq n
\end{split}
\end{gather*}
form an $\F$-basis of $\triangle_n$.
\end{lem}

\noindent By the symmetry of the defining relations of $\triangle$ there is an $\F$-algebra automorphism $\rho$ of $\triangle$ that sends
\begin{gather*}
(A,B,C,\alpha,\beta,\gamma)
\;\; \mapsto \;\;
(B,C,A,\beta,\gamma,\alpha).
\end{gather*}
Moreover \cite[Theorem~6.4]{uaw2011} showed that

\begin{lem}\label{lem:rho}
$\Omega^\rho=\Omega$.
\end{lem}

\subsection{An embedding of $\triangle$ into $\F[a^{\pm 1},b^{\pm1},c^{\pm 1}]\otimes_\F \U$}\label{s:uqsl2}

Assuming that $q^2\not=1$ the quantum group $\U$ is an $\F$-algebra generated by $e$, $k^{\pm 1}$, $f$ subject to the relations
\begin{gather*}
kk^{-1}=k^{-1}k=1,\\
ke=q^2ek, \qquad \quad
kf=q^{-2}fk, \\
ef-fe=\frac{k-k^{-1}}{q-q^{-1}}.
\end{gather*}
Let us abbreviate $U=\U$.
The generators $e$, $k^{\pm 1}$, $f$ are called the {\it Chevalley generators} of $U$. The Casimir element of $U$ is defined as
\begin{gather*}
ef+\frac{q^{-1} k+q k^{-1}}{(q-q^{-1})^2}.
\end{gather*}
Here we denote by $\Lambda$ a normalized Casimir element of $U$ obtained by multiplying the factor $(q-q^{-1})^2$.
The Poincar\'{e}--Birkhoff--Witt theorem for $U$ can be found in \cite[Theorem~1.5]{jantzen}:

\begin{lem}\label{lem:PBWU}
The monomials
\begin{gather*}
f^s k^{\pm i} e^r
\qquad \quad
\hbox{for all $i,r,s\in \N$}
\end{gather*}
form an $\F$-basis of $U$.
\end{lem}

\noindent For each $n\in \Z$ let $U_n$ denote the $\F$-subspace of $U$ spanned by
\begin{gather*}
f^s k^{\pm i} e^r \qquad \quad \hbox{for all $i,r,s\in \N$ with $r-s=n$}.
\end{gather*}
Recall from \cite[\S1.9]{jantzen} that the $\F$-subspaces $\{U_n\}_{n\in \Z}$ of $U$ satisfy
\setlist[enumerate,1]{leftmargin=2.4em}
\begin{enumerate}
\item[(G1)] $U=\bigoplus\limits_{n\in \Z} U_n$;

\item[(G2)] $U_m\cdot U_n\subseteq U_{m+n}$ for all $m, n \in \Z$.
\end{enumerate}
\setlist[enumerate,1]{leftmargin=2em}
In other words, the $\F$-spaces $\{U_n\}_{n\in \Z}$ give a $\Z$-gradation of $U$.
By (G1) any element $u\in U$ can be uniquely written as
\begin{gather*}
\sum_{n\in \Z} u_n
\end{gather*}
where $u_n\in U_n$ and almost all zero. For each $n\in \Z$ the element $u_n$ is called the {\it homogeneous component of $u$ of degree $n$} or simply the {\it $n$-homogeneous component of $u$}.

Consider the elements
\begin{gather*}
x=k^{-1}-q^{-1}(q-q^{-1})ek^{-1}, \qquad \quad
y^{\pm 1}=k^{\pm 1},\qquad \quad
z=k^{-1}+(q-q^{-1})f.
\end{gather*}
Solving for $e$, $k^{\pm1}$, $f$ it yields that
\begin{gather*}
e=\frac{q(1-xy)}{q-q^{-1}}, \qquad \quad
k^{\pm 1}=y^{\pm 1}, \qquad \quad
f=\frac{z-y^{-1}}{q-q^{-1}}.
\end{gather*}
Thus $x$, $y^{\pm 1}$, $z$ generate $U$ and they are called the {\it equitable generators} of $U$ \cite[Definition~2.2]{equit2005}. Let $a$, $b$, $c$ denote three commuting indeterminates over $\F$. Extending the ground field $\F$ of $U$ to the Laurent polynomial ring $\F[a^{\pm 1},b^{\pm 1},c^{\pm 1}]$, a recent work \cite[Theorem~2.16--2.18]{uaw&equit2011} of Terwilliger showed that

\begin{lem}\label{lem:naturalABC}
There exists a unique $\F$-algebra injective homomorphism $\natural:\triangle\to \F[a^{\pm 1},b^{\pm 1},c^{\pm 1}]\otimes_\F U$ with
\begin{eqnarray*}
A^\natural &=&
ax+a^{-1}y+qbc^{-1}(1-xy),
\\
B^\natural &=&
by+b^{-1}z+qca^{-1}(1-yz),
\\
C^\natural &=&
cz+c^{-1}x+qab^{-1}(1-zx),
\\
\alpha^\natural &=&
(b+b^{-1})(c+c^{-1})+(a+a^{-1})\Lambda,
\\
\beta^\natural &=&
(c+c^{-1})(a+a^{-1})+(b+b^{-1})\Lambda,
\\
\gamma^\natural &=&
(a+a^{-1})(b+b^{-1})+(c+c^{-1})\Lambda,\\
\Omega^\natural
&=&
(a+a^{-1})^2+(b+b^{-1})^2+(c+c^{-1})^2+
(a+a^{-1})(b+b^{-1})(c+c^{-1})\Lambda
+\Lambda^2-2.
\end{eqnarray*}
\end{lem}

Let $U'$ denote the $\F$-subalgebra of $U$ generated by $x$, $y$, $z$. In terms of the equitable generators
\begin{gather*}
\Lambda=q x
+q^{-1} y
+q z
-q x y z.
\end{gather*}
Hence $\Lambda\in U'$. By Lemma~\ref{lem:naturalABC} the image of $\natural$ is contained in $\F[a^{\pm 1},b^{\pm 1},c^{\pm 1}]\otimes_\F U'$. More precisely,

\begin{lem}\label{lem:image}
$\F[\alpha,\beta,\gamma,\Omega]^\natural$ is contained in $\F[a^{\pm 1},b^{\pm1},c^{\pm 1}]\otimes_\F \F[\Lambda]$.
\end{lem}

\noindent In \cite[Lemma~10.11 and Proposition~10.14]{uaw&equit2011} the $\F$-algebra automorphism $\rho:\triangle\to \triangle$ below Lemma~\ref{lem:PBW2} is extended as follows.

\begin{lem}\label{lem:tilderho}
There exists a unique $\F$-algebra automorphism $\widetilde \rho$ of $\F[a^{\pm 1},b^{\pm 1},c^{\pm 1}]\otimes_\F U'$ that sends
\begin{gather*}
(a,b,c,x,y,z)
\;\; \mapsto \;\;
(b,c,a,y,z,x).
\end{gather*}
Moreover the diagram
\begin{table}[H]
\centering
\begin{tikzpicture}
\matrix(m)[matrix of math nodes,
row sep=2.6em, column sep=2.8em,
text height=1.5ex, text depth=0.25ex]
{
\triangle
&\F[a^{\pm 1},b^{\pm 1},c^{\pm 1}]\otimes_\F U'\\
\triangle
&\F[a^{\pm 1},b^{\pm 1},c^{\pm 1}]\otimes_\F U'\\
};
\path[->,font=\scriptsize,>=angle 90]
(m-1-1) edge node[auto] {$\natural$} (m-1-2)
(m-1-1) edge node[left] {$\rho$} (m-2-1)
(m-2-1) edge node[auto] {$\natural$} (m-2-2)
(m-1-2) edge node[auto] {$\widetilde\rho$} (m-2-2);
\end{tikzpicture}
\end{table}
\noindent
commutes.
\end{lem}

\noindent By \cite[Lemma~10.13]{uaw&equit2011} we have

\begin{lem}\label{lem:tilderhoLambda}
$\Lambda^{\widetilde \rho}=\Lambda$.
\end{lem}

\subsection{Gaussian binomial coefficients and Chebyshev polynomials}\label{s:qbin}

The main purpose of this subsection is to display two identities to facilitate the evaluation of $T_{\dbar}(A)^\natural$, $T_{\dbar}(B)^\natural$, $T_{\dbar}(C)^\natural$.
Let $Q$ denote an indeterminate over the rational number field. Recall the notation
\begin{gather*}
[n]=\frac{Q^n-Q^{-n}}{Q-Q^{-1}} \qquad \quad
\hbox{for all $n\in \Z$}
\end{gather*}
and the Gaussian binomial coefficients
\begin{gather*}
{n\brack i}=\prod_{j=1}^{i}\frac{[n-j+1]}{[j]}
\qquad \quad
\hbox{for all $i\in \N$ and $n\in \Z$}.
\end{gather*}
The Gaussian binomial coefficients are shown to be contained in $\Z[Q^{\pm 1}]$. Since $\lim\limits_{Q\to 1}[n]=n$ for $n\in \Z$ the binomial coefficients are a limit case of the Gaussian binomial coefficients. The binomial formula can be generalized as follows.

\begin{prop}\label{prop:qbino}
If two elements $R$, $S$ in a $\Z[Q^{\pm 1}]$-algebra satisfy
$SR=Q^2 RS$, then
\begin{gather*}
(R+S)^n=\sum_{i=0}^n {n \brack i} Q^{i(n-i)}
R^{n-i} S^i
\qquad \quad \hbox{for all $n\in \N$}.
\end{gather*}
In particular, if $Q$ is a primitive $d^{\, th}$ root of unity then
$(R+S)^{\dbar}=R^{\dbar}+S^{\dbar}$.
\end{prop}

Set $\N^*=\N\setminus\{0\}$. By \cite[\S9.8.2]{Koe2010} the normalized Chebyshev polynomials of the first kind $\{T_n(X)\}_{n\in \N}$ given in (\ref{e:Tn}) satisfy
\begin{gather*}
X T_n(X)=T_{n+1}(X)+T_{n-1}(X) \qquad \quad \hbox{for $n\in \N^*$}
\end{gather*}
with $T_0(X)=2$ and $T_1(X)=X$.

\begin{lem}\label{lem:Tchara}
\begin{enumerate}
\item $
T_n(X+X^{-1})=X^n+X^{-n}
$
for all $n\in \N$.

\item $T_m(T_n(X))=T_{mn}(X)$ for all $m$, $n\in \N$.
\end{enumerate}
\end{lem}
\begin{proof}
To see (i)  proceed by a routine induction on $n$ and apply the above recurrence relation.
To see (ii) replace $X$ by $X+X^{-1}$ in $T_m(T_n(X))$ and apply (i).
\end{proof}

\noindent
Recall the central elements $\Gamma_0$, $\Gamma_1$, $\Gamma_2$ of $\V$ and the embedding $\flat$ of $\V$ from Introduction. To evaluate the images of $\Gamma_0$, $\Gamma_1$, $\Gamma_2$ under $\flat$ Iorgov established the identity \cite[Lemma~1]{iog02}:

\begin{prop}\label{prop:iorg}
If three elements $R^{\pm 1}$, $S$ in a $\Z[Q^{\pm 1}]$-algebra satisfy the relations
$R R^{-1}=R^{-1} R=1$ and $SR=Q^2 RS$, then $T_n(R+S+R^{-1})$ is equal to 
\begin{gather*}
R^n+R^{-n}+\sum_{i=1}^n \sum_{j=0}^{n-i}
 \frac{[n]}{[i]} {i+j-1\brack i-1}{n-j-1\brack i-1} Q^{i(n-i-2j)} R^{n-i-2j} S^i
\end{gather*}
for each $n\in \N$. In particular, if $Q$ is a primitive $d^{\, th}$ root of unity then
$T_{\dbar}(R+S+R^{-1})=R^{\dbar}+S^{\dbar}+R^{-\dbar}$.
\end{prop}

Henceforth each element in $\Z[Q^{\pm 1}]$ is also treated as an element in $\F$ via the ring homomorphism $\Z[Q^{\pm 1}]\to \F$ given by $1\mapsto 1$ and $Q\mapsto q$.

\subsection{A presentation for $Z(U)$}\label{s:preZU}
In \cite[Theorem~4.2]{ck90} Concini and Kac showed that $Z(U)$ is the commutative $\F$-algebra generated by $e^{\dbar}$, $k^{\pm \dbar}$, $f^{\dbar}$, $\Lambda$ subject to the relations $k^{\dbar} k^{-\dbar}=1$ and
\begin{gather}\label{e:TLambda}
T_{\dbar}(\Lambda)=
(q-q^{-1})^{2\dbar} e^{\dbar} f^{\dbar}
+q^{\dbar} (k^{\dbar}+k^{-\dbar}).
\end{gather}
Note that $\dbar$ could be any positive integer except $1$ because $q^2\not=1$ is necessary in the definition of $U$.

\begin{lem}\label{lem:basisZU}
For the tower
$
\F\subseteq \F[\Lambda]\subseteq Z(U)
$
the following hold:
\begin{enumerate}
\item  The elements
\begin{gather*}
\Lambda^n
\qquad \quad \hbox{for all $n\in \N$}
\end{gather*}
form an $\F$-basis of $\F[\Lambda]$.

\item
The elements
\begin{gather*}
k^{\pm \dbar i} f^{\dbar j}, \qquad
k^{\pm \dbar i}, \qquad
k^{\pm \dbar i} e^{\dbar j}
\qquad \quad
\hbox{for all $i\in \N$ and $j\in \N^*$}
\end{gather*}
form an $\F[\Lambda]$-basis of $Z(U)$.
\end{enumerate}
\end{lem}
\begin{proof}
Clearly $\F[\Lambda]$ is spanned by $\{\Lambda^n\}_{n\in \N}$ over $\F$. An induction on $n\in \N$ yields that
\begin{gather}\label{e:Lambdan}
\Lambda^n=(q-q^{-1})^{2n}f^n e^n+\sum_{i=0}^{n-1}f^i k_i e^i
\end{gather}
where $k_i\in \F[k^{\pm 1}]$. By Lemma~\ref{lem:PBWU} the elements $\{\Lambda^n\}_{n\in \N}$ are linearly independent over $\F$. This shows (i).

Relation (\ref{e:Lambdan}) along with Lemma~\ref{lem:PBWU} also implies that the elements
\begin{gather*}
\Lambda^n k^{\pm \dbar i}f^{\dbar j}, \qquad
\Lambda^n k^{\pm \dbar i}, \qquad
\Lambda^n k^{\pm \dbar i} e^{\dbar j}
\qquad \quad \hbox{for all $i,n\in \N$ and $j\in \N^*$}
\end{gather*}
are linearly independent over $\F$. It follows that
$k^{\pm \dbar i} f^{\dbar j}$,
$k^{\pm \dbar i}$,
$k^{\pm \dbar i} e^{\dbar j}$
for all $i\in \N$ and $j\in \N^*$ are linearly independent over $\F[\Lambda]$. Since $Z(U)$ is generated by  $e^{\dbar}$, $k^{\pm \dbar}$, $f^{\dbar}$, $\Lambda$ the monomials
\begin{gather*}
f^{\dbar r} k^{\pm \dbar i} e^{\dbar s}
\qquad \quad
\hbox{for all $i,r,s\in \N$}
\end{gather*}
span $Z(U)$ over $\F[\Lambda]$. Applying (\ref{e:TLambda}) each of these monomials is an $\F[\Lambda]$-linear combination of $k^{\pm \dbar i} f^{\dbar j}$,
$k^{\pm \dbar i}$,
$k^{\pm \dbar i} e^{\dbar j}$
for all $i\in \N$ and $j\in \N^*$. This shows (ii).
\end{proof}

\section{Three central elements of $\triangle$ at roots of unity}\label{s:3centralABC}

\subsection{The elements $T_{\dbar}(A)$, $T_{\dbar}(B)$, $T_{\dbar}(C)$ central in $\triangle$}\label{s:3central}

Observe that the number $\dbar$ is chosen to be the order of the root of unity $q^2$. Let $Y$ denote an indeterminate over $\F$ commuting with $X$. Consider the Laurent polynomials
\begin{gather*}
\Theta_i(Y)=q^{2i}Y^{-1} + q^{-2i}Y
\qquad \quad \hbox{for all $i\in \Z$}.
\end{gather*}

\begin{lem}\label{lem:theta}
For all $i$, $j\in \Z$
\begin{enumerate}
\item $\Theta_i(Y)=\Theta_j(Y)$ if and only if  $\dbar$ divides $i-j$;

\item $T_{\dbar}(\Theta_i(Y))=Y^{\dbar}+Y^{-\dbar}$.
\end{enumerate}
\end{lem}
\begin{proof}
(i) follows by the factorization $\Theta_i(Y)-\Theta_j(Y)=(q^{i-j}-q^{j-i}) (q^{i+j} Y^{-1}-q^{-i-j} Y)$.
To see (ii) apply Lemma~\ref{lem:Tchara}(i).
\end{proof}

By virtue of Lemma~\ref{lem:theta} one may prove that

\begin{thm}\label{thm:3central}
The elements $T_{\dbar}(A)$, $T_{\dbar}(B)$, $T_{\dbar}(C)$ are  central in $\triangle$.
\end{thm}
\begin{proof}
Without loss we show that $T_{\dbar}(A)$ is central in $\triangle$ since the $\F$-algebra automorphism $\rho:\triangle\to \triangle$ cyclically permutes $A$, $B$, $C$. The $\F$-algebra $\triangle$ is generated by $A$, $B$, $C$ and three central elements  $\alpha$, $\beta$, $\gamma$. Evidently $T_{\dbar}(A)$ commutes with $A$. Thus it suffices to show that $T_{\dbar}(A)$ commutes with $B$ and $C$. To evaluate $B\,T_{\dbar}(A)$ we begin to express $BA^n$ ($n\in \N$) as a linear combination of the $\F$-basis of $\triangle$ given in Lemma~\ref{lem:PBW}:

\begin{lem}\label{lem:PQRS}
For each $n\in \N$
\begin{gather*}
BA^n = P_n(A)B+Q_n(A)C+R_n(A)\beta+S_n(A)\gamma
\end{gather*}
where $P_n(X)$, $Q_n(X)$, $R_n(X)$, $S_n(X)\in \F[X]$ are recurrently defined by $P_0(X)=1$,
$Q_0(X)=0$,
$R_0(X)=0$,
$S_0(X)=0$ and for $n\in \N^*$
\begin{eqnarray}
P_n(X)&=&
q^2XP_{n-1}(X)
-q^{-1}(q^2-q^{-2})Q_{n-1}(X),
\label{P->Q}\\
Q_n(X)&=&
q^{-2}XQ_{n-1}(X)
+q(q^2-q^{-2})P_{n-1}(X),
\label{Q->P}\\
R_n(X)&=&
X R_{n-1}(X)
+q^{-1}(q-q^{-1})Q_{n-1}(X),
\label{R->Q}\\
S_n(X)&=&
X S_{n-1}(X)
-q(q-q^{-1}) P_{n-1}(X).
\label{S->P}
\end{eqnarray}
\end{lem}
\begin{proof}
Proceed by induction on $n$. There is nothing to prove for $n=0$. By induction hypothesis
\begin{gather}\label{e:BAn}
BA^n=P_{n-1}(A)BA+Q_{n-1}(A)CA+R_{n-1}(A)A\beta+S_{n-1}(A)A\gamma
\qquad \quad
\hbox{for $n\geq 1$}.
\end{gather}
This lemma follows by substituting (\ref{e:uaw3}), (\ref{e:uaw1}) into (\ref{e:BAn}).
\end{proof}
\noindent Applying Lemma~\ref{lem:PQRS} a direct calculation yields that
\begin{align*}
&P_1(X)=q^2 X,
\qquad
Q_1(X)=q(q^2-q^{-2}),
\qquad
R_1(X)=0,
\qquad
S_1(X)=-q(q-q^{-1}),\\
&P_2(X)=q^4 X^2-(q^2-q^{-2})^2,
\qquad \quad \,
Q_2(X)=q(q^4-q^{-4}) X,\\
&R_2(X)=(q-q^{-1})(q^2-q^{-2}),
\qquad \quad
S_2(X)=-q^2(q^2-q^{-2})X.
\end{align*}
In particular, if $q^2=1$ then $P_1(X)=X$, $Q_1(X)=0$, $R_1(X)=0$, $S_1(X)=0$ and if $q^4=1$ then
$P_2(X)=X^2$, $Q_2(X)=0$, $R_2(X)=0$, $S_2(X)=0$.
Thus this theorem holds for $\dbar=1$, $2$.

For the rest assume that $\dbar>2$.
It seems difficult to find the closed forms of $P_n(X)$, $Q_n(X)$, $R_n(X)$, $S_n(X)$ for any $n\in \N$. Instead we evaluate these polynomials at $\{\Theta_i(Y)\}_{i\in \Z}$.

\begin{lem}\label{lem:PQRStheta}
For all $i\in \Z$ and $n\in \N$
\begin{eqnarray*}
P_n(\Theta_i(Y))&=&
\frac{\Theta_{i+1}(Y)-q^2\Theta_i(Y)}
{\Theta_{i+1}(Y)-\Theta_{i-1}(Y)}
\,\Theta_{i-1}(Y)^n
+\frac{q^2\Theta_i(Y)-\Theta_{i-1}(Y)}
{\Theta_{i+1}(Y)-\Theta_{i-1}(Y)}
\,\Theta_{i+1}(Y)^n,\\[2pt]
Q_n(\Theta_i(Y))&=&
q(q^2-q^{-2})\,
\frac{\Theta_{i+1}(Y)^n-\Theta_{i-1}(Y)^n}
{\Theta_{i+1}(Y)-\Theta_{i-1}(Y)},\\[2pt]
R_n(\Theta_i(Y)) &=&
\frac{(q-q^{-1}) (q^2-q^{-2})}{\Theta_{i+1}(Y)-\Theta_{i-1}(Y)}
\left(
\frac{\Theta_{i+1}(Y)^n-\Theta_i(Y)^n}{\Theta_{i+1}(Y)-\Theta_i(Y)}
-
\frac{\Theta_i(Y)^n-\Theta_{i-1}(Y)^n}{\Theta_i(Y)-\Theta_{i-1}(Y)}
\right), \\[2pt]
S_n(\Theta_i(Y))&=&
\frac{q(q-q^{-1})}{\Theta_{i+1}(Y)-\Theta_{i-1}(Y)}
\bigg(
\frac{\Theta_{i+1}(Y)^n-\Theta_i(Y)^n}
{\Theta_{i+1}(Y)-\Theta_i(Y)}\,
(
\Theta_{i-1}(Y)-q^2\Theta_i(Y)
)
\\
&&\quad-\;\;
\frac{\Theta_i(Y)^n-\Theta_{i-1}(Y)^n}
{\Theta_i(Y)-\Theta_{i-1}(Y)}\,
(
\Theta_{i+1}(Y)-q^2\Theta_i(Y)
)\bigg).
\end{eqnarray*}
\end{lem}
\begin{proof}
Multiply (\ref{P->Q}), (\ref{Q->P}) by $q^{-2} X$, $-q^{-1}(q^2-q^{-2})$ respectively. The difference of the resulting equations gives
\begin{gather*}
q^{-2} XP_n(X)+q^{-1}(q^2-q^{-2}) Q_n(X)=
(X^2+(q^2-q^{-2})^2) P_{n-1}(X)
\qquad \quad \hbox{for $n\in\N^*$}.
\end{gather*}
Subtracting the above from the equation obtained by replacing the index $n$ in (\ref{P->Q}) by $n+1$, it follows that the polynomials $\{P_n(X)\}_{n\in \N}$ satisfy
\begin{gather*}
P_{n+1}(X)=(q^2+q^{-2})X P_n(X)-(X^2+(q^2-q^{-2})^2) P_{n-1}(X) \qquad \quad \hbox{for $n\in \N^*$.}
\end{gather*}
Fix $i\in \Z$. Substituting $X=\Theta_i(Y)$ into the above relation, the characteristic polynomial of the resulting recurrence relation is
\begin{gather*}
K(X)=X^2-(q^2+q^{-2})\Theta_i(Y)X+\Theta_i(Y)^2+(q^2-q^{-2})^2.
\end{gather*}
It is straightforward to verify that $\Theta_{i-1}(Y)$ and $\Theta_{i+1}(Y)$ are two roots of $K(X)$. Under the assumption $\dbar>2$ the Laurent polynomials $\Theta_{i-1}(Y)$, $\Theta_{i+1}(Y)$ are distinct by Lemma~\ref{lem:theta}(i). Therefore $P_n(\Theta_i(Y))$ is of the form
\begin{gather*}
\Xi(Y)\Theta_{i-1}(Y)^n+\Sigma(Y)\Theta_{i+1}(Y)^n
\qquad \quad
\hbox{for all $n\in \N$},
\end{gather*}
where $\Xi(Y)$, $\Sigma(Y)\in \F(Y)$.
The functions $\Xi(Y)$, $\Sigma(Y)$ can be solved from
the matrix equation
\begin{gather*}
\left(
\begin{array}{cc}
1 &1\\
\Theta_{i-1}(Y) &\Theta_{i+1}(Y)
\end{array}
\right)
\left(
\begin{array}{c}
\Xi(Y)\\
\Sigma(Y)
\end{array}
\right)
=
\left(
\begin{array}{c}
1\\
q^2\Theta_i(Y)
\end{array}
\right),
\end{gather*}
which obtains from $P_0(X)=1$ and $P_1(X)=q^2 X$.
The identity for $Q_n(\Theta_i(Y))$ follows by a similar argument.

Proceed by induction on $n$ to verify the identities for $R_n(\Theta_i(Y))$ and $S_n(\Theta_i(Y))$. It holds for $n=0$ since $R_0(X)=0$ and $S_0(X)=0$. For $n\geq 1$ apply induction hypotheses to (\ref{R->Q}), (\ref{S->P}) with $X=\Theta_i(Y)$ and simplify them by using the identities for $Q_{n-1}(\Theta_i(Y))$, $P_{n-1}(\Theta_i(Y))$ respectively.
\end{proof}

As a consequence of Lemma~\ref{lem:PQRS} the element $B T_n(A)$ is equal to
\begin{gather*}
\mathcal P_n(A) B+
\mathcal Q_n(A) C+
\mathcal R_n(A) \beta+
\mathcal S_n(A) \gamma
\qquad
\quad
\hbox{for each $n\in \N$},
\end{gather*}
where $\mathcal P_n(X)$, $\mathcal Q_n(X)$, $\mathcal R_n(X)$, $\mathcal S_n(X)\in \F[X]$ are obtained from (\ref{e:Tn}) by replacing $X^{n-2i}$ by $P_{n-2i}(X)$, $Q_{n-2i}(X)$, $R_{n-2i}(X)$, $S_{n-2i}(X)$ for all $0\leq i \leq \lfloor n/2\rfloor$, respectively. To see that $T_{\dbar}(A)$ commutes with $B$ it needs to show that
\begin{gather*}
\mathcal P_{\dbar}(X)=T_{\dbar}(X), \qquad \quad
\mathcal Q_{\dbar}(X)=\mathcal R_{\dbar}(X)=\mathcal S_{\dbar}(X)=0.
\end{gather*}
By Lemma~\ref{lem:PQRS} a routine induction yields that  $P_n(X)$ is of degree $n$ with leading coefficient $q^{2n}$ and hence so is $\mathcal P_n(X)$. In particular $\mathcal P_{\dbar}(X)$ is a monic polynomial of degree $\dbar$ as well as $T_{\dbar}(X)$. Applying the formula of $P_n(\Theta_i(Y))$ in Lemma~\ref{lem:PQRStheta} a direct calculation yields that
\begin{gather*}
\mathcal P_n(\Theta_i(Y))=\frac{\Theta_{i+1}(Y)-q^2\Theta_i(Y)}
{\Theta_{i+1}(Y)-\Theta_{i-1}(Y)}
\;T_n(\Theta_{i-1}(Y))
+\frac{q^2\Theta_i(Y)-\Theta_{i-1}(Y)}
{\Theta_{i+1}(Y)-\Theta_{i-1}(Y)}
\;T_n(\Theta_{i+1}(Y))
\end{gather*}
for all $i\in \Z$. Simplifying the above equality by using Lemma~\ref{lem:theta}(ii) it follows that
\begin{gather*}
\mathcal P_{\dbar}(\Theta_i(Y))=Y^{\dbar}+Y^{-\dbar}
\qquad \quad
\hbox{for all $i\in \Z$}.
\end{gather*}
So far we have seen that $\mathcal P_{\dbar}(X)$ and $T_{\dbar}(X)$ are monic polynomials of degree $\dbar$ and agree on $\{\Theta_i(Y)\}^{\dbar-1}_{i=0}$ which are pairwise distinct by Lemma~\ref{lem:theta}(i). Therefore $\mathcal P_{\dbar}(X)=T_{\dbar}(X)$ by division algorithm. A similar argument shows that each of $\mathcal Q_{\dbar}(X)$, $\mathcal R_{\dbar}(X)$, $\mathcal S_{\dbar}(X)$ is identically zero. By a similar argument $T_{\dbar}(A)$ commutes with $C$. The result follows.
\end{proof}

As consequences of Theorem~\ref{thm:3central} the following elements are contained in $Z(\triangle)$.

\begin{prop}
\begin{enumerate}
\item $T_n(A)$, $T_n(B)$, $T_n(C)$ are cental in $\triangle$ for all nonnegative multiples $n$ of $\dbar$.
\item
$\prod\limits_{i=0}^{\dbar-1} (A-\Theta_i(\lambda))$,
$\prod\limits_{i=0}^{\dbar-1} (B-\Theta_i(\lambda))$,
$\prod\limits_{i=0}^{\dbar-1} (C-\Theta_i(\lambda))$
are central in $\triangle$ for all nonzero $\lambda\in \F$.
\end{enumerate}
\end{prop}
\begin{proof}
(i) is immediate from Lemma~\ref{lem:Tchara}(ii). By the division algorithm  Lemma~\ref{lem:theta}(i) implies that
\begin{gather}\label{e:zeroT}
T_{\dbar}(X)=\prod\limits_{i=0}^{\dbar-1} (X-\Theta_i(Y))+Y^{\dbar}+Y^{-\dbar}.
\end{gather}
(ii) is immediate from (\ref{e:zeroT}).
\end{proof}

\noindent Equation (\ref{e:zeroT}) also implies that the zeros of $T_{\dbar}(X)$ are
$\{\Theta_i(\lambda)\}_{i=0}^{\dbar-1}$ where $\lambda$ is a root of $Y^{\dbar}+Y^{-\dbar}$. More generally, please refer to \cite{AWpolyzero:97} for zeros of the Askey-Wilson polynomials at roots of unity.

\subsection{The images of $T_{\dbar}(A)$, $T_{\dbar}(B)$, $T_{\dbar}(C)$ under $\natural$}\label{s:3embedd}
Throughout this subsection assume that $q^2\not=1$.
 Recall the embedding $\natural:\triangle\to \F[a^{\pm 1},b^{\pm1},c^{\pm 1}]\otimes_\F U$ from Lemma~\ref{lem:naturalABC}. This subsection is devoted to computing $T_{\dbar}(A)^\natural$, $T_{\dbar}(B)^\natural$, $T_{\dbar}(C)^\natural$. To do this we need the following conversion formulae between $e^{\dbar}$, $k^{\pm \dbar}$, $f^{\dbar}$ and
$x^{\dbar}$, $y^{\pm \dbar}$, $z^{\dbar}$.

\begin{lem}\label{lem:ZU}
Assume that $q^2\not=1$. Then the following hold:
\begin{enumerate}
\item The elements
$x^{\dbar}$, $y^{\pm \dbar}$, $z^{\dbar}$ are equal to
\begin{gather*}
k^{-\dbar}-q^{\dbar} (q-q^{-1})^{\dbar} e^{\dbar} k^{-\dbar},
\qquad \quad
k^{\pm \dbar},
\qquad \quad
k^{-\dbar}+(q-q^{-1})^{\dbar} f^{\dbar}
\end{gather*}
respectively. In particular $x^{\dbar}$, $y^{\pm \dbar}$, $z^{\dbar}$ are central in $U$.

\item The elements $e^{\dbar}$, $k^{\pm \dbar}$, $f^{\dbar}$ are equal to
\begin{gather*}\label{e:genZ(U)2}
q^{\dbar}(q-q^{-1})^{-\dbar}(1-x^{\dbar} y^{\dbar}),
\qquad \quad
y^{\pm \dbar},
\qquad \quad
(q-q^{-1})^{-\dbar}(z^{\dbar}-y^{-\dbar})
\end{gather*}
respectively.
\end{enumerate}
\end{lem}
\begin{proof}
To get (ii) solve (i) for $e^{\dbar}$, $k^{\pm \dbar}$, $f^{\dbar}$. Thus it only needs to prove (i). Since $y^{\pm 1}=k^{\pm 1}$ the equality $y^{\pm \dbar}=k^{\pm \dbar}$ holds. Recall that
\begin{gather*}
x=k^{-1}-q^{-1}(q-q^{-1})ek^{-1}.
\end{gather*}
The first two defining relations of $U$ imply that
\begin{gather}\label{e:ek-1}
ek^{-1}=q^2 k^{-1} e.
\end{gather}
Applying Proposition~\ref{prop:qbino} with $(R,S)=(k^{-1},-q^{-1}(q-q^{-1})ek^{-1})$ it follows that
\begin{gather*}\label{e:xhbar}
x^{\dbar}=
k^{-\dbar}
+(-1)^{\dbar}q^{-\dbar}(q-q^{-1})^{\dbar}(ek^{-1})^{\dbar}.
\end{gather*}
To further simplify the addend in the above identity, one may use the following lemmas:
\begin{lem}\label{lem:ekn}
$(ek^{-1})^n=q^{n(n-1)} e^n k^{-n}$ for all $n\in \N$.
\end{lem}
\begin{proof}
This lemma follows by inductively applying (\ref{e:ek-1}).
\end{proof}

\begin{lem}\label{lem:qh}
$q^{\dbar(\dbar-1)}=(-1)^{\dbar-1}$.
\end{lem}
\begin{proof}
By the setting of $\dbar$ the left-hand side is equal to $1$ (resp. $-1$) if $\dbar$ is odd (resp. even).
\end{proof}
\noindent
By a similar argument the identity for $z^{\dbar}$ follows.
\end{proof}

We are ready to evaluate the images of $T_{\dbar}(A)$, $T_{\dbar}(B)$, $T_{\dbar}(C)$ under $\natural$.

\begin{thm}\label{thm:embedd}
Assume that $q^2\not=1$.
Then $T_{\dbar}(A)^\natural$, $T_{\dbar}(B)^\natural$, $T_{\dbar}(C)^\natural$ are central in $\F[a^{\pm 1},b^{\pm 1},c^{\pm 1}]\otimes_\F U$. Moreover $T_{\dbar}(A)^\natural$, $T_{\dbar}(B)^\natural$, $T_{\dbar}(C)^\natural$ are equal to
\begin{gather}
a^{\dbar} x^{\dbar} + a^{-\dbar}  y^{\dbar}
+q^{\dbar} b^{\dbar} c^{-\dbar} (1-x^{\dbar} y^{\dbar}),
\label{e:TA}\\
b^{\dbar}  y^{\dbar} + b^{-\dbar} z^{\dbar}
+q^{\dbar} c^{\dbar} a^{-\dbar}(1-y^{\dbar} z^{\dbar}), \label{e:TB}\\
c^{\dbar}  z^{\dbar} + c^{-\dbar}  x^{\dbar}
+q^{\dbar} a^{\dbar} b^{-\dbar}(1-z^{\dbar} x^{\dbar}), \label{e:TC}
\end{gather}
respectively. Their nonzero homogeneous components are as follows:
\begin{table}[H]
\begin{minipage}[H]{0.5\textwidth}
\centering
\extrarowheight=3.3pt
\begin{tabular}{c|c}
degree &$T_{\dbar}(A)^\natural$\\

\midrule[2pt]

$0$
&$a^{\dbar}k^{-\dbar}+a^{-\dbar}k^{\dbar}$\\

\midrule[1pt]

$\dbar$
&$(q-q^{-1})^{\dbar}(b^{\dbar} c^{-\dbar}-q^{\dbar} a^{\dbar} k^{-\dbar}) e^{\dbar}$
\end{tabular}
\end{minipage}%
\begin{minipage}[H]{0.5\textwidth}
\centering
\extrarowheight=3.3pt
\begin{tabular}{c|c}
degree &$T_{\dbar}(B)^\natural$\\

\midrule[2pt]

$-\dbar$
&$(q-q^{-1})^{\dbar}(b^{-\dbar}-q^{\dbar} a^{-\dbar}c^{\dbar}k^{\dbar})f^{\dbar}$\\

\midrule[1pt]

$0$
&$b^{-\dbar}k^{-\dbar}+b^{\dbar}k^{\dbar}$
\end{tabular}
\end{minipage}
\end{table}

\begin{table}[H]
\centering
\extrarowheight=3.3pt
\begin{tabular}{c|c}
degree &$T_{\dbar}(C)^\natural$\\

\midrule[2pt]

$-\dbar$
&$(q-q^{-1})^{\dbar}(c^{\dbar}-q^{\dbar} a^{\dbar}b^{-\dbar}k^{-\dbar}
)f^{\dbar}$\\

\midrule[1pt]

$0$
&$-q^{\dbar}a^{\dbar}b^{-\dbar}k^{-2\dbar}
+
(c^{\dbar}+c^{-\dbar})k^{-\dbar}
+a^{\dbar}b^{-\dbar}(q-q^{-1})^{2\dbar} f^{\dbar} k^{-\dbar} e^{\dbar}+q^{\dbar}a^{\dbar}b^{-\dbar}$\\

\midrule[1pt]

$\dbar$
&$(q-q^{-1})^{\dbar}(a^{\dbar} b^{-\dbar}k^{-\dbar}-q^{\dbar} c^{-\dbar} )k^{-\dbar}
 e^{\dbar}$
\end{tabular}
\end{table}
\end{thm}
\begin{proof}
Express $A^\natural$ given in Lemma~\ref{lem:naturalABC} in terms of Chevalley generators:
\begin{gather*}
A^\natural=
ak^{-1}+a^{-1}k+(q-q^{-1})(bc^{-1}e-a q^{-1} ek^{-1}).
\end{gather*}
By relation (\ref{e:ek-1}),
Proposition~\ref{prop:iorg} is applicable to $(R,S)=(ak^{-1},(q-q^{-1})(bc^{-1}e-a q^{-1}ek^{-1}))$. It follows that
\begin{gather*}
T_{\dbar}(A)^\natural=a^{\dbar}k^{-\dbar}+
a^{-\dbar} k^{\dbar}+
(q-q^{-1})^{\dbar}(bc^{-1}e-a q^{-1} ek^{-1})^{\dbar}.
\end{gather*}
By (G2) the $0$- and $\dbar$-homogeneous components of $T_{\dbar}(A)^\natural$ are equal to
\begin{gather*}
a^{\dbar}k^{-\dbar}+
a^{-\dbar} k^{\dbar},
\qquad \quad
(q-q^{-1})^{\dbar}(bc^{-1}e-a q^{-1} ek^{-1})^{\dbar}
\end{gather*}
respectively and the other homogeneous components of $T_{\dbar}(A)^\natural$ are zero.
To show the $\dbar$-homogeneous component of $T_{\dbar}(A)$ as stated, one may simplify
$
(bc^{-1}e-a q^{-1} ek^{-1})^{\dbar}
$
by applying Proposition~\ref{prop:qbino} with $(R,S)=(bc^{-1}e,-a q^{-1} ek^{-1})$ followed by using Lemma~\ref{lem:ekn} with $n=\dbar$ and Lemma~\ref{lem:qh}.

To see that $T_{\dbar}(A)^\natural$ is equal to (\ref{e:TA}), make use of Lemma~\ref{lem:ZU}(ii) to express the homogeneous components of $T_{\dbar}(A)^\natural$ in terms of $x^{\dbar}$, $y^{\pm \dbar}$, $z^{\dbar}$.
The identities (\ref{e:TB}), (\ref{e:TC}) now follow by applying the automorphism $\widetilde \rho$ of $\F[a^{\pm 1},b^{\pm 1},c^{\pm 1}]\otimes_\F U$ given in Lemma~\ref{lem:tilderho} to (\ref{e:TA}) and (\ref{e:TB}),  respectively. Apply Lemma~\ref{lem:ZU}(i) to get the homogeneous components of $T_{\dbar}(B)^\natural$ and $T_{\dbar}(C)^\natural$.
\end{proof}

We remark that Theorem~\ref{thm:embedd} provides an alternative proof of Theorem~\ref{thm:3central} by the injectivity of $\natural$. Nevertheless this proof is lack of intuition.

\section{The center of $\triangle$ at roots of unity}\label{s:ZDelta}

\subsection{A basis of $Z(\triangle)$}\label{s:FbasisZ}

As seen in Theorem~\ref{thm:3central} the elements $T_{\dbar}(A)$, $T_{\dbar}(B)$, $T_{\dbar}(C)$ are central in $\triangle$. In this subsection we shall show that

\begin{thm}\label{thm:basisZDelta}
For each $n\in \N$
the elements
\begin{gather*}
T_{\dbar}(A)^{i_0} T_{\dbar}(B)^{i_1} T_{\dbar}(C)^{i_2}
\alpha^{j_0} \beta^{j_1} \gamma^{j_2}
\Omega^\ell,
\qquad \quad
\begin{split}
&i_0,i_1,i_2,j_0,j_1,j_2,\ell\in \N, \qquad i_0i_1i_2=0,\\
&\dbar(i_0+i_1+i_2)+j_0+j_1+j_2+3\ell\leq n
\end{split}
\end{gather*}
form an $\F$-basis of $Z(\triangle)\cap \triangle_n$.
In particular
the elements
\begin{gather*}
T_{\dbar}(A)^{i_0} T_{\dbar}(B)^{i_1} T_{\dbar}(C)^{i_2}
\alpha^{j_0} \beta^{j_1} \gamma^{j_2}
\Omega^\ell
 \qquad \quad
\hbox{for all $i_0,i_1,i_2,j_0,j_1,j_2,\ell\in \N$ \; with \; $i_0 i_1 i_2=0$}
\end{gather*}
form an $\F$-basis of $Z(\triangle)$.
\end{thm}

To prove Theorem~\ref{thm:basisZDelta} we resort to the $\F$-bases of $\{\triangle_n\}_{n\in \N}$ as below. For each $n\in \N$ define
$\I_n$ to be the set of all $10$-tuples $(r_0,r_1,r_2,i_0,i_1,i_2,j_0,j_1,j_2,\ell)$ of nonnegative integers with
\begin{gather*}
r_0+r_1+r_2+\dbar(i_0+i_1+i_2)+j_0+j_1+j_2+3\ell\leq n,
\\
(r_0+i_0)(r_1+i_1)(r_2+i_2)=0,
\qquad \quad
r_0, r_1, r_2<\dbar.
\end{gather*}

\begin{lem}\label{lem:basisDeltan_n-1}
For each $n\in \N$
the elements
\begin{gather*}
A^{r_0} B^{r_1} C^{r_2}
T_{\dbar}(A)^{i_0} T_{\dbar}(B)^{i_1} T_{\dbar}(C)^{i_2}
\alpha^{j_0} \beta^{j_1} \gamma^{j_2} \Omega^\ell
\qquad \quad \hbox{for all $(r_0,r_1,r_2,i_0,i_1,i_2,j_0,j_1,j_2,\ell)\in \I_n$}
\end{gather*}
form an $\F$-basis of $\triangle_n$.
\end{lem}
\begin{proof}
Regard $\F[X]$ as a positional system with base $T_{\dbar}(X)$. Namely, view each $P(X)\in \F[X]$ as
$\sum_{i\in \N} P_i(X)T_{\dbar}(X)^i$ with $P_i(X)\in \F[X]$ of degree less than $\dbar$ and almost all zero. This lemma follows by applying the positional system to the $\F$-basis of $\triangle_n$ given in Lemma~\ref{lem:PBW2}.
\end{proof}

\begin{lem}\label{lem:uawmod}
The following equations hold:
\begin{eqnarray*}\label{e:uawmod}
CB&=&q^{2} BC
\qquad \pmod{\triangle_1},
\\
CA&=&q^{-2} AC
\qquad \pmod{\triangle_1},
\\
BA&=&q^{2} AB
\qquad \pmod{\triangle_1}.
\end{eqnarray*}
\end{lem}
\begin{proof}
Immediate from (\ref{e:uaw2})--(\ref{e:uaw1}).
\end{proof}

\noindent As usual, for any $r$, $s$ in a ring the notation $[r,s]$ is used to denote the commutator $rs-sr$.

\begin{lem}\label{lem:Omega-ABC}
\begin{enumerate}
\item For all $i_0$, $i_1$, $i_2\in \N$
\begin{align*}
&[A^{i_0} B^{i_1} C^{i_2},A]
=(q^{2(i_1-i_2)}-1)
A^{i_0+1} B^{i_1} C^{i_2}
\qquad \pmod{\triangle_{i_0+i_1+i_2}},
\\
&[A^{i_0} B^{i_1} C^{i_2},B]
=(q^{2i_2}-q^{2i_0}) A^{i_0} B^{i_1+1} C^{i_2}
\qquad \pmod{\triangle_{i_0+i_1+i_2}},
\\
&[A^{i_0} B^{i_1} C^{i_2},C]
=(1-q^{2(i_1-i_0)})A^{i_0}B^{i_1} C^{i_2+1}
\qquad \pmod{\triangle_{i_0+i_1+i_2}}.
\end{align*}

\item For all $i_0$, $i_1$, $i_2\in \N^*$
\begin{gather*}
A^{i_0}B^{i_1}C^{i_2}
= (-1)^\ell q^{\ell(\ell-2i_1)}
A^{i_0-\ell}
B^{i_1-\ell}
C^{i_2-\ell}
\Omega^\ell
\qquad
\pmod{\triangle_{i_0+i_1+i_2-1}}
\end{gather*}
where $\ell$ is any nonnegative integer less than or equal to each of $i_0$, $i_1$, $i_2$.
\end{enumerate}
\end{lem}
\begin{proof}
Property (F2) will be used frequently in the proof without reference. (i) follows by applying Lemma~\ref{lem:uawmod}. Suppose $i_0$, $i_1$, $i_2\in \N^*$. By Lemma~\ref{lem:uawmod} it yields that
\begin{gather*}\label{e:Omega-ABC}
A^{i_0}B^{i_1}C^{i_2}
=q^{-2(i_1-1)}
A^{i_0-1}
B^{i_1-1}
C^{i_2-1}
 A B C
\qquad
\pmod{\triangle_{i_0+i_1+i_2-1}}.
\end{gather*}
Observe from (\ref{e:Omega}) that $ABC=-q^{-1}\Omega\,\bmod{\triangle_2}$.
Therefore
\begin{gather*}
A^{i_0}B^{i_1}C^{i_2}
=-q^{1-2i_1}
A^{i_0-1}
B^{i_1-1}
C^{i_2-1}
\Omega
\qquad
\pmod{\triangle_{i_0+i_1+i_2-1}}.
\end{gather*}
Now (ii) follows by a routine induction on $i_0+i_1+i_2$.
\end{proof}

It is ready to prove Theorem~\ref{thm:basisZDelta}.

\medskip

\noindent {\it Proof of Theorem~\ref{thm:basisZDelta}.}
The second statement is immediate from the first statement by (F1).
Proceed by induction on $n\in \N$ to prove the first statement.
It is clear for $n=0$. Now assume $n\geq 1$. Pick any $R\in Z(\triangle)\cap\triangle_n$.
By Lemma~\ref{lem:basisDeltan_n-1} there exist unique $c(r_0,r_1,r_2,i_0,i_1,i_2,j_0,j_1,j_2,\ell)\in \F$ for all $(r_0,r_1,r_2,i_0,i_1,i_2,j_0,j_1,j_2,\ell)\in \I_n\setminus \I_{n-1}$ such that $R\bmod\triangle_{n-1}$ is equal to
\begin{gather*}
S=
%\!\!\!\!\!\!
\sum_{(r_0,r_1,\ldots,\ell)}
c(r_0,r_1,r_2,i_0,i_1,i_2,j_0,j_1,j_2,\ell)
A^{r_0} B^{r_1} C^{r_2}
T_{\dbar}(A)^{i_0} T_{\dbar}(B)^{i_1} T_{\dbar}(C)^{i_2}
\alpha^{j_0} \beta^{j_1} \gamma^{j_2}
\Omega^\ell
\end{gather*}
where the sum is over all $(r_0,r_1,r_2,i_0,i_1,i_2,j_0,j_1,j_2,\ell)\in \I_n\setminus \I_{n-1}$.
We shall see that if at least one of $r_0$, $r_1$, $r_2$ is positive then
$c(r_0,r_1,r_2,i_0,i_1,i_2,j_0,j_1,j_2,\ell)=0$.
If this is true then $S\in Z(\triangle)$ by Theorem~\ref{thm:3central} and this theorem follows by applying induction hypothesis to $R-S\in Z(\triangle)\cap \triangle_{n-1}$. It remains to prove the claim.

Pick any $c(r_0,r_1,r_2,i_0,i_1,i_2,j_0,j_1,j_2,\ell)$ with at least one of $r_0$, $r_1$, $r_2$ in $\N^*$. Without loss assume that $r_0\in \N^*$. By the definition of $\I_n$, at least one of $(r_1,i_1)$, $(r_2,i_2)$ is equal to $(0,0)$, say $(r_1,i_1)$. In this case, consider the commutator $[S,C]$. Since $R$ is central in $\triangle$ the commutator $[R,C]=0$. It follows from (F2) that
$$
[S,C]\in \triangle_n.
$$
On the other hand, applying Lemma~\ref{lem:Omega-ABC} the coefficient of
$
A^{r_0+\dbar i_0} C^{r_2+\dbar i_2+1}
\alpha^{j_0} \beta^{j_1} \gamma^{j_2}
\Omega^\ell
$
in $[S,C]$ with respect to the $\F$-basis of $\triangle_{n+1}$
given in Lemma~\ref{lem:PBW2} is equal to
\begin{gather*}\label{e:coeff}
(1-q^{-2r_0})\cdot c(r_0,r_1,r_2,i_0,i_1,i_2,j_0,j_1,j_2,\ell).
\end{gather*}
Since $[S,C]\in \triangle_n$ the above is equal to $0$. Recall that
the order of the root of unity $q^2$ is equal to $\dbar$. Therefore  $q^{2r_0}\not=1$ and this forces that the claim is true. The result follows.
\hfill $\square$

\subsection{A relation in $Z(\triangle)$}\label{s:relZDelta}

The goal of this subsection is to prove Theorems~\ref{thm:poly} and \ref{thm:relZDelta} stated below.
For each $n\in \N$ let
\begin{eqnarray*}
\phi_n(X_0,X_1,X_2;X) &=& T_n(X) T_n(X_0)+T_n(X_1) T_n(X_2),\\
\psi_n(X_0,X_1,X_2;X) &=& T_{2n}(X)+T_n(X_0)^2+T_n(X_1)^2+T_n(X_2)^2+
T_n(X)T_n(X_0)T_n(X_1)T_n(X_2).
\end{eqnarray*}
To each $i\in \Z/3\Z$ we associate an $\F[X]$-algebra automorphism $\ubar i$ of $\F[X_0,X_1,X_2,X]$ with
\begin{gather*}
X^{\ubar i}_j= X_{i+j}
\qquad \quad
\hbox{for all $j\in \Z/3\Z$}.
\end{gather*}

\begin{thm}\label{thm:poly}
For each $n\in \N$ there exist unique polynomials $\Phi_n(X_0,X_1,X_2;X)$, $\Psi_n(X_0,X_1,X_2;X)$ over $\Z$ satisfying
\begin{gather*}
\Phi_n(
\phi_m^{\ubar 0},
\phi_m^{\ubar 1},
\phi_m^{\ubar 2};
\psi_m)
=
\phi_{mn},
\qquad \quad
\Psi_n(
\phi_m^{\ubar 0},
\phi_m^{\ubar 1},
\phi_m^{\ubar 2};
\psi_m)
=
\psi_{mn}
\qquad \quad
\hbox{for all $m\in \N$}.
\end{gather*}
\end{thm}

\begin{thm}\label{thm:relZDelta}
The relation
\begin{align*}
&q^{\dbar}
\Phi_{\dbar}^{\ubar 0}(\alpha,\beta,\gamma;\Omega) T_{\dbar}(A)
+q^{\dbar}
\Phi_{\dbar}^{\ubar 1}(\alpha,\beta,\gamma;\Omega) T_{\dbar}(B)
+q^{\dbar}
\Phi_{\dbar}^{\ubar 2}(\alpha,\beta,\gamma;\Omega) T_{\dbar}(C)
\\
&\quad=\;\;
q^{\dbar} T_{\dbar}(A) T_{\dbar}(B) T_{\dbar}(C)
+T_{\dbar}(A)^2
+ T_{\dbar}(B)^2
+ T_{\dbar}(C)^2
+ \Psi_{\dbar}(\alpha,\beta,\gamma;\Omega)
-2
\end{align*}
holds in $Z(\triangle)$.
\end{thm}

Before launching into the proof of Theorems~\ref{thm:poly} and \ref{thm:relZDelta} we make some preparations.

\begin{lem}\label{lem:leading}
For any $r_0$, $r_1$, $r_2$, $s \in \N$, if there are $i_0$, $i_1$, $i_2$, $j\in\N$
such that the degree of
\begin{gather*}
(\phi_1^{\ubar 0})^{i_0}
(\phi_1^{\ubar 1})^{i_1}
(\phi_1^{\ubar 2})^{i_2}
(\psi_1)^{j}
\end{gather*}
as a polynomial of $X$ is equal to $r_0+r_1+r_2+2s$ with leading coefficient
$
X_0^{r_0}X_1^{r_1}X_2^{r_2}
$, then $(i_0,i_1,i_2,j)=(r_0,r_1,r_2,s)$.
\end{lem}
\begin{proof}
Expand $(\phi_1^{\ubar 0})^{i_0}
(\phi_1^{\ubar 1})^{i_1}
(\phi_1^{\ubar 2})^{i_2}
(\psi_1)^j$ directly.
\end{proof}

\begin{lem}\label{lem:Tijnew}
For each $n\in \N$ there exists a unique $\F$-algebra isomorphism $\F[X_0,X_1,X_2,X]\to \F[T_n(X_0),T_n(X_1),T_n(X_2),T_n(X)]$ that sends
\begin{gather*}
X_0
\;\;\mapsto\;\;
T_n(X_0),
\qquad
X_1
\;\;\mapsto\;\;
T_n(X_1),
\qquad
X_2
\;\;\mapsto\;\;
T_n(X_2),
\qquad
X
\;\;\mapsto\;\;
T_n(X).
\end{gather*}
Moreover this isomorphism maps $\phi_m$ to $\phi_{mn}$ and $\psi_m$ to $\psi_{mn}$ for all $m\in \N$.
\end{lem}
\begin{proof}
Lemma~\ref{lem:Tchara}(i) implies the algebraic independence of $T_n(X_0)$, $T_n(X_1)$, $T_n(X_2)$, $T_n(X)$ over $\F$. Therefore the desired isomorphism exists.
Applying Lemma~\ref{lem:Tchara}(ii) it yields that
\begin{eqnarray*}
\phi_m(T_n(X_0),T_n(X_1),T_n(X_2);T_n(X))
&=&\phi_{mn}(X_0,X_1,X_2;X),\\
\psi_m(T_n(X_0),T_n(X_1),T_n(X_2);T_n(X))
&=&\psi_{mn}(X_0,X_1,X_2;X).
\end{eqnarray*}
This lemma follows.
\end{proof}

\begin{lem}\label{lem:indp4}
For each $n\in \N^*$ the polynomials $\phi_n^{\ubar 0}$, $\phi_n^{\ubar 1}$, $\phi_n^{\ubar 2}$, $\psi_n$ are algebraically independent over $\F$.
\end{lem}
\begin{proof}
Consider the equation
\begin{gather}\label{e:indp1}
\sum_{i_0,i_1,i_2,j\in \N}
c_{i_0,i_1,i_2,j}
(\phi_1^{\ubar 0})^{i_0}
(\phi_1^{\ubar 1})^{i_1}
(\phi_1^{\ubar 2})^{i_2}
(\psi_1)^j=0
\end{gather}
with $c_{i_0,i_1,i_2,j}\in \F$ and almost all zero. Suppose there exist  $r_0,r_1,r_2,s\in \N$ with $c_{r_0,r_1,r_2,s}\not=0$ and other coefficients $c_{i_0,i_1,i_2,j}\not=0$ only if $r_0+r_1+r_2+2s\geq i_0+i_1+i_2+2j$.
By Lemma~\ref{lem:leading} the coefficient of $X_0^{r_0}X_1^{r_1}X_2^{r_2}X^{r_0+r_1+r_2+2s}$ in the left-hand side of (\ref{e:indp1}) is equal to $c_{r_0,r_1,r_2,s}$, a contradiction to (\ref{e:indp1}). Therefore $c_{i_0,i_1,i_2,j}=0$ for all $i_0,i_1,i_2,j\in \N$. This shows that $\phi_1^{\ubar 0}$, $\phi_1^{\ubar 1}$, $\phi_1^{\ubar 2}$, $\psi_1$ are algebraically independent over $\F$. For any $n\in \N^*$ the algebraic independence of  $\phi_n^{\ubar 0}$, $\phi_n^{\ubar 1}$, $\phi_n^{\ubar 2}$, $\psi_n$ now follows from Lemma~\ref{lem:Tijnew}.
\end{proof}

\begin{prop}\label{prop:TABC}
There exist two unique polynomials $\Phi(X_0,X_1,X_2;X)$, $\Psi(X_0,X_1,X_2;X)$ over $\F$ satisfying
\begin{gather*}
\Phi(\phi_1^{\ubar 0},
\phi_1^{\ubar 1},
\phi_1^{\ubar 2};
\psi_1)=
\phi_{\dbar},
\qquad \quad
\Psi(\phi_1^{\ubar 0},
\phi_1^{\ubar 1},
\phi_1^{\ubar 2};
\psi_1)=
\psi_{\dbar}.
\end{gather*}
Moreover
\begin{align*}
&
q^{\dbar} \Phi^{\ubar 0}(\alpha,\beta,\gamma;\Omega) T_{\dbar}(A)
+q^{\dbar} \Phi^{\ubar 1}(\alpha,\beta,\gamma;\Omega) T_{\dbar}(B)
+q^{\dbar} \Phi^{\ubar 2}(\alpha,\beta,\gamma;\Omega) T_{\dbar}(C)
\\
&\quad=\;\;q^{\dbar}
T_{\dbar}(A) T_{\dbar}(B) T_{\dbar}(C)
+T_{\dbar}(A)^2
+ T_{\dbar}(B)^2
+ T_{\dbar}(C)^2
+\Psi(\alpha,\beta,\gamma;\Omega)
-2.
\end{align*}
\end{prop}
\begin{proof}
By Lemma~\ref{lem:indp4} the polynomials $\Phi(X_0,X_1,X_2;X)$ and $\Psi(X_0,X_1,X_2;X)$ are unique if they exist. To see the existence consider the central element
\begin{gather*}
T_{\dbar}(A) T_{\dbar}(B) T_{\dbar}(C).
\end{gather*}
By Lemma~\ref{lem:Omega-ABC}(ii) with $i_0=i_1=i_2=\dbar$, it follows that
$T_{\dbar}(A) T_{\dbar}(B) T_{\dbar}(C)-(-1)^{\dbar}q^{-\dbar^2}\Omega^{\dbar}\in \triangle_{3\dbar-1}$.
By Theorem~\ref{thm:basisZDelta} there are unique $\Sigma_A$, $\Sigma_B$, $\Sigma_C$, $\Xi_A$, $\Xi_B$, $\Xi_C$, $\Phi_A$, $\Phi_B$, $\Phi_C$, $\Psi\in \F[\alpha,\beta,\gamma,\Omega]$ such that $T_{\dbar}(A) T_{\dbar}(B) T_{\dbar}(C)$ is equal to
\begin{align}\label{e:TABC}
\begin{split}
&\Sigma_A T_{\dbar}(B)T_{\dbar}(C)
+\Sigma_B T_{\dbar}(C)T_{\dbar}(A)
+\Sigma_C T_{\dbar}(A)T_{\dbar}(B)
+\Xi_A T_{\dbar}(A)^2+\Xi_B T_{\dbar}(B)^2+\Xi_C T_{\dbar}(C)^2
\\
&\quad+\;\;
\Phi_A T_{\dbar}(A)
+\Phi_B T_{\dbar}(B)
+\Phi_C T_{\dbar}(C)
+\Psi.
\end{split}
\end{align}
For $\dbar=1$ it follows from (\ref{e:Omega}) that
\begin{align*}
&\Sigma_A=\Sigma_B=\Sigma_C=0,
\qquad \quad
\Xi_A=\Xi_B=\Xi_C=-q,
\\
&\Phi_A=\alpha,
\qquad
\Phi_B=\beta,
\qquad
\Phi_C=\gamma,
\qquad
\Psi=q(2-\Omega).
\end{align*}
Thus $\Phi(X_0,X_1,X_2;X)=X_0$ and $\Psi(X_0,X_1,X_2;X)=X$ when $\dbar=1$.

Assume that $\dbar>1$. By Lemma~\ref{lem:basisZU}(ii), $\F[a^{\pm 1},b^{\pm 1},c^{\pm 1}]\otimes_\F Z(U)$ has the
$\F[a^{\pm 1},b^{\pm 1},c^{\pm 1}]\otimes_\F \F[\Lambda]$-basis
\begin{gather}\label{e:basis}
k^{\pm \dbar i} f^j, \qquad
k^{\pm \dbar i},\qquad
k^{\pm \dbar i} e^j
\qquad \quad \hbox{for all $i\in \N$ and $j\in \N^*$}.
\end{gather}
By Lemma~\ref{lem:image} the images of $\Sigma_A$, $\Sigma_B$, $\Sigma_C$, $\Xi_A$, $\Xi_B$, $\Xi_C$, $\Phi_A$, $\Phi_B$, $\Phi_C$, $\Psi$ under $\natural$ are contained in $\F[a^{\pm 1},b^{\pm 1},c^{\pm 1}]\otimes_\F \F[\Lambda]$. To find the values of them, express
\begin{gather*}
T_{\dbar}(A)^\natural, \qquad
T_{\dbar}(B)^\natural, \qquad
T_{\dbar}(C)^\natural, \qquad
 T_{\dbar}(A)^{2\natural}, \qquad
T_{\dbar}(B)^{2\natural}, \qquad
T_{\dbar}(C)^{2\natural},\\
T_{\dbar}(A)^\natural T_{\dbar}(B)^\natural, \qquad
T_{\dbar}(B)^\natural T_{\dbar}(C)^\natural, \qquad
T_{\dbar}(C)^\natural T_{\dbar}(A)^\natural, \qquad
T_{\dbar}(A)^\natural T_{\dbar}(B)^\natural T_{\dbar}(C)^\natural
\end{gather*}
as linear combinations of (\ref{e:basis}). The results of calculation are included in the appendix of this paper. The coefficient of $e^{2\dbar}$ in $T_{\dbar}(A)^\natural T_{\dbar}(B)^\natural T_{\dbar}(C)^\natural$ is equal to
\begin{gather*}
-(q-q^{-1})^{2\dbar} q^{\dbar}  b^{2\dbar} c^{-2\dbar}.
\end{gather*}
Denote by $(\ref{e:TABC})^\natural$ the equation obtained by applying $\natural$ to $(\ref{e:TABC})$. Therefore the coefficient of $e^{2\dbar}$ in $(\ref{e:TABC})^\natural$ is equal to
\begin{gather*}
(q-q^{-1})^{2\dbar}b^{2\dbar} c^{-2\dbar}\,\Xi_A^\natural.
\end{gather*}
This leads to $\Xi_A=-q^{\dbar}$. The comparison of the coefficients of $k^{-\dbar}e^{2\dbar}$
yields that $\Sigma_B=0$. Since $T_{\dbar}(A) T_{\dbar}(B) T_{\dbar}(C)$ is fixed by $\rho$ the automorphism $\rho$ sends
\begin{eqnarray*}
(\Sigma_A,\Sigma_B,\Sigma_C) &\mapsto& (\Sigma_B,\Sigma_C,\Sigma_A),
\\
(\Xi_A,\Xi_B,\Xi_C) &\mapsto& (\Xi_B,\Xi_C,\Xi_A),
\\
\qquad
(\Phi_A,\Phi_B,\Phi_C) &\mapsto& (\Phi_B,\Phi_C,\Phi_A).
\end{eqnarray*}
As a result
\begin{gather*}
\Sigma_A=\Sigma_B=\Sigma_C=0,
\qquad \qquad
\Xi_A=\Xi_B=\Xi_C=-q^{\dbar}.
\end{gather*}
Now, the comparison of the coefficients of $e^{\dbar}$ in $T_{\dbar}(A)^\natural T_{\dbar}(B)^\natural T_{\dbar}(C)^\natural$ and $(\ref{e:TABC})^\natural$ determines that
\begin{eqnarray*}
\Phi_A^\natural&=&T_{\dbar}(\Lambda)(a^{\dbar}+a^{-\dbar})
+(b^{\dbar}+b^{-\dbar})(c^{\dbar}+c^{-\dbar}).
\end{eqnarray*}
By Lemma~\ref{lem:tilderho} the automorphism $\widetilde \rho$ of $\F[a^{\pm 1},b^{\pm 1},c^{\pm 1}]\otimes_\F U'$ cyclically permutes $\Phi_A^\natural$, $\Phi_B^\natural$, $\Phi_C^\natural$. Combined with Lemma \ref{lem:tilderhoLambda} it follows that
\begin{eqnarray*}
\Phi_B^\natural&=&
T_{\dbar}(\Lambda)(b^{\dbar}+b^{-\dbar})
+(c^{\dbar}+c^{-\dbar})(a^{\dbar}+a^{-\dbar}),\\
\Phi_C^\natural&=&
T_{\dbar}(\Lambda)(c^{\dbar}+c^{-\dbar})
+(a^{\dbar}+a^{-\dbar})(b^{\dbar}+b^{-\dbar}).
\end{eqnarray*}
Comparing the coefficients of $1$ it follows that $\Psi^\natural$ is equal to $-q^{\dbar}$ times
\begin{gather*}
T_{2\dbar}(\Lambda)
+(a^{\dbar}+a^{-\dbar})^2
+(b^{\dbar}+b^{-\dbar})^2
+(c^{\dbar}+c^{-\dbar})^2
+T_{\dbar}(\Lambda)
(a^{\dbar}+a^{-\dbar})
(b^{\dbar}+b^{-\dbar})
(c^{\dbar}+c^{-\dbar})-2.
\end{gather*}

Normalize $\Psi$ by multiplying the factor $-q^{\dbar}$ followed by adding $2$. Abuse the notation $\Psi$ to denote the normalized one. By Lemma~\ref{lem:Tchara}(i) we may write
\begin{eqnarray*}
\Phi_A^\natural&=&
\phi_{\dbar}(a+a^{-1},b+b^{-1},c+c^{-1};\Lambda),\\
\Psi^\natural &=&
\psi_{\dbar}(a+a^{-1},b+b^{-1},c+c^{-1};\Lambda).
\end{eqnarray*}
Since $\Phi_A$, $\Psi\in \F[\alpha,\beta,\gamma;\Omega]$ there exist two polynomials $\Phi(X_0,X_1,X_2;X)$, $\Psi(X_0,X_1,X_2;X)$ over $\F$ with $\Phi(\alpha^\natural,\beta^\natural,\gamma^\natural;\Omega^\natural)
=\Phi_A^\natural$ and $\Psi(\alpha^\natural,\beta^\natural,\gamma^\natural;\Omega^\natural)=\Psi^\natural$. Observe from Lemma~\ref{lem:naturalABC} that
\begin{eqnarray*}
\alpha^\natural &=&
\phi_1^{\ubar 0}(a+a^{-1},b+b^{-1},c+c^{-1};\Lambda),
\\
\beta^\natural &=&
\phi_1^{\ubar 1}(a+a^{-1},b+b^{-1},c+c^{-1};\Lambda),
\\
\gamma^\natural &=&
\phi_1^{\ubar 2}(a+a^{-1},b+b^{-1},c+c^{-1};\Lambda),
\\
\Omega^\natural &=&
\psi_1(a+a^{-1},b+b^{-1},c+c^{-1};\Lambda).
\end{eqnarray*}
By Lemma~\ref{lem:basisZU}(i), $\F[a^{\pm 1},b^{\pm 1},c^{\pm 1}]\otimes_\F \F[\Lambda]$ has the $\F[a^{\pm 1},b^{\pm 1},c^{\pm 1}]$-basis
$\{\Lambda^n\}_{n\in \N}$.
In other words $\Lambda$ is algebraically independent over $\F[a^{\pm 1},b^{\pm 1},c^{\pm 1}]$. Together with the algebraic independence of $a+a^{-1}$, $b+b^{-1}$, $c+c^{-1}$ over $\F$, the elements
\begin{gather*}
a+a^{-1}, \qquad
b+b^{-1}, \qquad
c+c^{-1}, \qquad
\Lambda
\end{gather*}
are algebraically independent over $\F$. Concluding from the above results
the two polynomials $\Phi(X_0,X_1,X_2;X)$, $\Psi(X_0,X_1,X_2;X)$ satisfy the desired properties.
\end{proof}

It is in the position to prove Theorems~\ref{thm:poly} and \ref{thm:relZDelta}.

\medskip

\noindent {\it Proof of Theorem~\ref{thm:poly}.}
Set $\F=\C$ and consider the polynomials $\Phi(X_0,X_1,X_2;X)$, $\Psi(X_0,X_1,X_2;X)$ from Proposition~\ref{prop:TABC}.

Suppose some coefficients of $\Phi(X_0,X_1,X_2;X)$ are not integers.
Choose a monomial $X_0^{r_0}X_1^{r_1}X_2^{r_2}X^s$ in $\Phi(X_0,X_1,X_2;X)$ with non-integral coefficient that satisfies the property: For any $i_0$, $i_1$, $i_2$, $j\in \N$, if the coefficient of $X_0^{i_0}X_1^{i_1}X_2^{i_2}X^j$ in $\Phi(X_0,X_1,X_2;X)$ is not an integer, then $r_0+r_1+r_2+2s\geq i_0+i_1+i_2+2j$. Substitute $(X_0,X_1,X_2,X)=(\phi_1^{\ubar 0},\phi_1^{\ubar 1},\phi_1^{\ubar 2},\psi_1)$ into $\Phi(X_0,X_1,X_2;X)$.
By Lemma~\ref{lem:leading} the coefficient of $X_0^{r_0}X_1^{r_1}X_2^{r_2}X^{r_0+r_1+r_2+2s}$
in $\Phi(\phi_1^{\ubar 0},\phi_1^{\ubar 1},\phi_1^{\ubar 2};\psi_1)$ is not an integer. This contradicts that $\Phi(
\phi_1^{\ubar 0},\phi_1^{\ubar 1},\phi_1^{\ubar 2};\psi_1)=\phi_{\dbar}$ is a polynomial over $\Z$. Therefore the coefficients of $\Phi(X_0,X_1,X_2;X)$ are integers. Let $m\in \N$ be given. Substitute
$X_i=T_m(X_i)$ for all $i\in\Z/3\Z$ and $X=T_m(X)$ into $\Phi(
\phi_1^{\ubar 0},\phi_1^{\ubar 1},\phi_1^{\ubar 2};\psi_1)=\phi_{\dbar}$.
By Lemma~\ref{lem:Tijnew} it follows that
$
\Phi(
\phi_m^{\ubar 0},\phi_m^{\ubar 1},\phi_m^{\ubar 2};\psi_m)
=
\phi_{m\dbar}$. We have shown that $\Phi(X_0,X_1,X_2;X)=\Phi_{\dbar}(X_0,X_1,X_2;X)$. By a similar argument $\Psi(X_0,X_1,X_2;X)=\Psi_{\dbar}(X_0,X_1,X_2;X)$.

Given any positive integer $d$ there always exists an element $q\in \F$ to be a primitive $d^{\rm\, th}$ root of unity since $\F$ is now assumed to be $\C$. Therefore $\dbar$ could be any positive integer. The result follows.
\hfill $\square$

\medskip

\noindent{\it Proof of Theorem~\ref{thm:relZDelta}.}
By Theorem~\ref{thm:poly} the polynomials $\Phi_{\dbar}(X_0,X_1,X_2;X)$, $\Psi_{\dbar}(X_0,X_1,X_2;X)$ satisfy $
\Phi_{\dbar}
(\phi_1^{\ubar 0},
\phi_1^{\ubar 1},
\phi_1^{\ubar 2};
\psi_1)
=\phi_{\dbar}$, 	
$\Psi_{\dbar}
(\phi_1^{\ubar 0},
\phi_1^{\ubar 1},
\phi_1^{\ubar 2};
\psi_1)
=\psi_{\dbar}$.
Thus, for any ground field $\F$
the polynomials $\Phi(X_0,X_1,X_2;X)$, $\Psi(X_0,X_1,X_2;X)$ from Proposition~\ref{prop:TABC} are equal to $\Phi_{\dbar}(X_0,X_1,X_2;X)$, $\Psi_{\dbar}(X_0,X_1,X_2;X)$ respectively.
\hfill $\square$

\medskip

Some properties about $\{\Phi_n\}_{n\in \N}$ and $\{\Psi_n\}_{n\in \N}$ are worthy of mention.

\begin{lem}\label{lem:poly1}
There exists a unique $\F$-algebra isomorphism $\F[X_0,X_1,X_2,X]\to \F[\phi_1^{\ubar 0}, \phi_1^{\ubar 1},\phi_1^{\ubar 2},\psi_1]$ that sends
\begin{gather*}
\Phi_n^{\ubar i}
\;\;
\mapsto
\;\;
\phi_n^{\ubar i},
\qquad \quad
\Psi_n^{\ubar i}
\;\;
\mapsto
\;\;
\psi_n
\qquad \quad
\hbox{for all $n\in \N$ and all $i\in \Z/3\Z$}.
\end{gather*}
\end{lem}
\begin{proof}
By construction $\psi_n=\psi_n^{\ubar i}$ for all $n\in \N$ and all $i\in \Z/3\Z$.
By Lemma~\ref{lem:indp4} there is a unique $\F$-algebra isomorphism $\F[X_0,X_1,X_2,X]\to \F[\phi_1^{\ubar 0}, \phi_1^{\ubar 1},\phi_1^{\ubar 2},\psi_1]$ that maps
$X_i$ to $\phi_1^{\ubar i}$
for all $i\in \Z/3\Z$
and
$X$ to $\psi_1$.
This isomorphism satisfies the desired property due to Theorem~\ref{thm:poly}.
\end{proof}

As consequences of Lemma~\ref{lem:poly1} we have

\begin{prop}\label{prop:poly2}
\begin{enumerate}
\item  $\Psi_n^{\ubar i}=\Psi_n$ for all $n\in \N$ and $i\in \Z/3\Z$.

\item
$
\Phi_n(\Phi_m^{\ubar 0},\Phi_m^{\ubar 1},\Phi_m^{\ubar 2};\Psi_m)=\Phi_{mn}
$
and
$
\Psi_n(\Phi_m^{\ubar 0},\Phi_m^{\ubar 1},\Phi_m^{\ubar 2};\Psi_m)=\Psi_{mn}
$
for all $m$, $n\in \N$.

\item $\Phi_n^{\ubar 0}$, $\Phi_n^{\ubar 1}$, $\Phi_n^{\ubar 2}$, $\Psi_n$ are algebraically independent over $\F$ for each $n\in \N^*$.
\end{enumerate}
\end{prop}
\begin{proof}
(i) is trivial. (ii) follows from Theorem~\ref{thm:poly}. (iii) follows from Lemma~\ref{lem:indp4}.
\end{proof}

\subsection{A presentation for $Z(\triangle)$}\label{s:preZDelta}

As a consequence of Theorem~\ref{thm:basisZDelta} the elements $T_{\dbar}(A)$, $T_{\dbar}(B)$, $T_{\dbar}(C)$, $\alpha$, $\beta$, $\gamma$, $\Omega$ form a set of generators of $Z(\triangle)$. This subsection is devoted to proving that

\begin{thm}\label{thm:presentation}
The center of $\triangle$ is the commutative $\F$-algebra generated by $T_{\dbar}(A)$, $T_{\dbar}(B)$, $T_{\dbar}(C)$, $\alpha$, $\beta$, $\gamma$, $\Omega$ subject to the relation
\begin{align*}
&
q^{\dbar} \Phi_{\dbar}^{\ubar 0}(\alpha,\beta,\gamma;\Omega) T_{\dbar}(A)
+q^{\dbar} \Phi_{\dbar}^{\ubar 1}(\alpha,\beta,\gamma;\Omega) T_{\dbar}(B)
+q^{\dbar} \Phi_{\dbar}^{\ubar 2}(\alpha,\beta,\gamma;\Omega) T_{\dbar}(C)
\\
&\quad=\;\;
q^{\dbar} T_{\dbar}(A) T_{\dbar}(B) T_{\dbar}(C)
+T_{\dbar}(A)^2
+ T_{\dbar}(B)^2
+ T_{\dbar}(C)^2
+\Psi_{\dbar}(\alpha,\beta,\gamma;\Omega)
-2.
\end{align*}
\end{thm}
\begin{proof}
Let $K$ denote the $\F$-subalgebra of $Z(\triangle)$ generated by $T_{\dbar}(A)$, $T_{\dbar}(B)$, $T_{\dbar}(C)$, $\alpha$, $\beta$, $\gamma$. Consider the $K$-algebra homomorphism $K[X]\to Z(\triangle)$ that maps $X\mapsto \Omega$. Since $Z(\triangle)=K[\Omega]$ by Theorem~\ref{thm:basisZDelta}, the homomorphism is surjective. By Theorem~\ref{thm:relZDelta} the kernel of this homomorphism contains the ideal of $K[X]$ generated by
\begin{align*}
P(X)&=
\Psi_{\dbar}(\alpha,\beta,\gamma;X)
-q^{\dbar}\Phi_{\dbar}^{\ubar 0}(\alpha,\beta,\gamma;X) T_{\dbar}(A)
-q^{\dbar}\Phi_{\dbar}^{\ubar 1}(\alpha,\beta,\gamma;X) T_{\dbar}(B)
-q^{\dbar}\Phi_{\dbar}^{\ubar 2}(\alpha,\beta,\gamma;X) T_{\dbar}(C)
\\
&\qquad+\;
q^{\dbar} T_{\dbar}(A) T_{\dbar}(B) T_{\dbar}(C)
+T_{\dbar}(A)^2
+ T_{\dbar}(B)^2
+ T_{\dbar}(C)^2
-2.
\end{align*}
It suffices to show the converse inclusion. Furthermore, we shall show that $P(X)$ is equal to the minimal polynomial $R(X)$ of $\Omega$ over the fraction field of $K$. We begin with a lemma which implies that $P(X)$ is a monic polynomial of degree $\dbar$.

\begin{lem}\label{lem:monic}
As polynomials of $X$,
\begin{enumerate}
\item $\Phi_n(X_0,X_1,X_2;X)$ is of degree less than or equal to $\lfloor n/2\rfloor$ for each $n\in \N$;

\item $\Psi_n(X_0,X_1,X_2;X)$ is monic of degree $n$ for each $n\in \N^*$.
\end{enumerate}
\end{lem}
\begin{proof}
By Theorem~\ref{thm:poly}, $
\Phi_n(\phi_1^{\ubar 0},\phi_1^{\ubar 1},\phi_1^{\ubar 2};\psi_1)=\phi_n$ is of degree $n$ as a polynomial of $X$.
Therefore, by Lemma~\ref{lem:leading} each monomial $X_0^{r_0} X_1^{r_1} X_2^{r_2} X^s$ in $\Phi_n(X_0,X_1,X_2;X)$ with nonzero coefficient satisfies $r_0+r_1+r_2+2s\leq n$ and in particular $2s\leq n$. This shows (i).

By Theorem~\ref{thm:poly},
$\Psi_n(\phi_1^{\ubar 0},
\phi_1^{\ubar 1},\phi_1^{\ubar 2};\psi_1)=\psi_n$
is monic of degree $2n$ as a polynomial of $X$. By Lemma~\ref{lem:leading} each monomial $X_0^{r_0} X_1^{r_1} X_2^{r_2} X^s$ in $\Psi_n(X_0,X_1,X_2;X)$ with nonzero coefficient satisfies $r_0+r_1+r_2+2s\leq 2n$ and the equality only holds for the monomial $X^n$ with coefficient $1$. This shows (ii).
\end{proof}

Recall that $\rho$ sends
\begin{eqnarray*}
(A,B,C,\alpha,\beta,\gamma) &\mapsto& (B,C,A,\beta,\gamma,\alpha).
\end{eqnarray*}
Hence $K$ is $\rho$-invariant. Here we abuse the notation $\rho$ to denote the $\F$-algebra automorphism $K[X]\to K[X]$ induced by $\rho$ that leaves $X$ fixed. By Proposition~\ref{prop:poly2}(i) a direct calculation yields that $P(X)$ is $\rho$-invariant.

\begin{lem}\label{lem:rhoPQR}
Let $Q(X)$ denote the quotient of $P(X)$ divided by $R(X)$. Then the coefficients of $Q(X)$, $R(X)$ lie in $K$ and each of $Q(X)$, $R(X)$ is $\rho$-invariant.
\end{lem}
\begin{proof}
By Lemma~\ref{lem:monic} the polynomial $P(X)$ is monic over $K$ and hence the roots of $P(X)$ are integral over $K$. Since $R(X)$ divides $P(X)$ the roots of $R(X)$ are integral over $K$ and so are the coefficients of $R(X)$.
By Lemma~\ref{lem:PBW}, $K$ is a six-variable polynomial ring over $\F$ generated by $T_{\dbar}(A)$, $T_{\dbar}(B)$, $T_{\dbar}(C)$, $\alpha$, $\beta$, $\gamma$.
In particular $K$ is integrally closed. Therefore $R(X)\in K[X]$ and so is $Q(X)$.

Since $\Omega^\rho=\Omega$ by Lemma~\ref{lem:rho}, the element $\Omega$ is a root of $R(X)^\rho$. Therefore $R(X)$ is $\rho$-invariant by the uniqueness of $R(X)$. Applying $\rho$ to $P(X)=Q(X)R(X)$ it follows that $Q(X)$ is $\rho$-invariant by the uniqueness of $Q(X)$.
\end{proof}

Let $L$ denote the $\F$-subalgebra of $K$ generated by $\alpha$, $\beta$, $\gamma$. For each $g\in K$ let $\deg g$ denote the degree of $g$ as a polynomial of $T_{\dbar}(A)$, $T_{\dbar}(B)$, $T_{\dbar}(C)$ over $L$.
Write a polynomial $G(X)\in K[X]$ as $\sum_{i\in \N} g_i X^i$ with all $g_i\in K$ and almost all zero. The coefficient $g_n$
is said to be {\it chief} if $n$ is the maximal nonnegative integer satisfying
\begin{gather*}
\deg g_i\leq \deg g_n \qquad \quad \hbox{for all $i\in \N$}.
\end{gather*}

\begin{lem}\label{lem:chief}
Given $F(X)$, $G(X)\in K[X]$, if the coefficients of $X^m$, $X^n$ in $F(X)$, $G(X)$ are chief with degrees $r$, $s$ respectively, then the coefficient of $X^{m+n}$ in $F(X)G(X)$ is chief with degree $r+s$.
\end{lem}
\begin{proof}
Write $F(x)=\sum\limits_{i\in \N} f_i X^i$ and $G(X)=\sum\limits_{i\in \N} g_i X^i$ with $f_i$, $g_i$ in $K$ and almost all zero. Then
\begin{gather*}
\deg f_i+\deg g_j\leq r+s
\qquad \quad \hbox{for all $i$, $j\in \N$}.
\end{gather*}
The equality holds if $i=m$ and $j=n$ and only if $i\leq m$ and $j\leq n$.
The result follows.
\end{proof}

From the definition of $P(X)$ we see that the constant term of $P(X)$ is chief with degree $3$. By Lemma~\ref{lem:chief} with $(F(X),G(X))=(Q(X),R(X))$ the constant term $q_0$ of $Q(X)$ is chief and so is the constant term $r_0$ of $R(X)$. Observe that the constant term of $P(X)$
is not equal to a product of two $\rho$-invariant elements in $K$ that have degrees $1$ and $2$. Therefore $\deg r_0$ is equal to $0$ or $3$ by Lemma~\ref{lem:rhoPQR}. Since $R(X)$ is of positive degree this forces that $\deg r_0=3$ and hence $\deg q_0=0$. This shows that $Q(X)$ is a constant. Since $P(X)$ and $R(X)$ are monic polynomials it follows that $P(X)=R(X)$.
\end{proof}

Replace the roles of $T_{\dbar}(A)$, $T_{\dbar}(B)$, $T_{\dbar}(C)$ in Theorem~\ref{thm:presentation} by $q^{\dbar} T_{\dbar}(A)$, $q^{\dbar} T_{\dbar}(B)$, $q^{\dbar} T_{\dbar}(C)$ respectively. Combining with Lemma~\ref{lem:Tijnew} and Lemma~\ref{lem:poly1} it follows that

\begin{thm}\label{thm:structureZDelta}
$Z(\triangle)$ is isomorphic to the polynomial ring $\F[X_0,X_1,X_2, X,Y_0,Y_1,Y_2]$ modulo the ideal generated by
\begin{gather}\label{e:mondrom}
\prod_{i\in \Z/3\Z} Y_i
+\sum_{i\in \Z/3\Z} Y_i^2
-\sum_{i\in \Z/3\Z}\phi_1^{\ubar i}(X_0,X_1,X_2;X) Y_i
+\psi_1(X_0,X_1,X_2;X)-2.
\end{gather}
\end{thm}

\noindent In the proof of Theorem~\ref{thm:presentation} we also see
that

\begin{thm}\label{thm:basisZDelta2}
The elements
\begin{gather*}
T_{\dbar}(A)^{i_0} T_{\dbar}(B)^{i_1} T_{\dbar}(C)^{i_2}
\alpha^{j_0} \beta^{j_1} \gamma^{j_2}
\Omega^\ell
\qquad \quad
\hbox{for all $i_0,i_1,i_2,j_0,j_1,j_2,\ell\in \N$
with $\ell<\dbar$}
\end{gather*}
form an $\F$-basis of $Z(\triangle)$.
\end{thm}

\begin{thm}\label{thm:integra1}
$Z(\triangle)$ is integral over the $\F$-subalgebra of $Z(\triangle)$ generated by $T_{\dbar}(A)$, $T_{\dbar}(B)$, $T_{\dbar}(C)$, $\alpha$, $\beta$, $\gamma$.
\end{thm}

We end this subsection with a feedback to $Z(\V)$.
By the construction of $\triangle$ there is an $\F$-algebra homomorphism $\iota:\triangle\to \V$ that sends
\begin{eqnarray*}
(A,B,C,\alpha,\beta,\gamma,\Omega) &\mapsto&
(
-(q^2-q^{-2})K_0,-(q^2-q^{-2})K_1,-(q^2-q^{-2})K_2,0,0,0,\Pi
).
\end{eqnarray*}
Recall from Introduction that $\Gamma_i$ denotes $T_{\dbar}(-(q^2-q^{-2})K_i)$ for each $i\in \Z/3\Z$. Hence $\iota$ maps $T_{\dbar}(A)$, $T_{\dbar}(B)$, $T_{\dbar}(C)$ to $\Gamma_0$, $\Gamma_1$, $\Gamma_2$ respectively.

From the definitions of $\phi_n(X_0,X_1,X_2;X)$, $\psi_n(X_0,X_1,X_2;X)$ we see that
\begin{eqnarray*}
\phi_n(0,0,0;X)&=&
2
\left(
\left\lfloor \frac{n}{2}\right\rfloor
-
\left\lceil \frac{n}{2}\right\rceil
+1
\right)
\left(
(-1)^{\lfloor n/2\rfloor}T_n(X)+2
\right),\\
\psi_n(0,0,0;X)&=&T_{2n}(X)+
4\left(
\left\lfloor \frac{n}{2}\right\rfloor
-
\left\lceil \frac{n}{2}\right\rceil
+1
\right)
\left(
(-1)^{\lfloor n/2\rfloor}2\,T_n(X)+3
\right)
\end{eqnarray*}
for all $n\in \N$.
Substitute $(X_0,X_1,X_2)=(0,0,0)$ into $\Phi_n(\phi_1^{\ubar 0},\phi_1^{\ubar 1},\phi_1^{\ubar 2};\psi_1)$. By Theorem~\ref{thm:poly} it follows that
\begin{gather*}
\Phi_n(0,0,0;T_2(X))=
2
\left(
\left\lfloor \frac{n}{2}\right\rfloor
-
\left\lceil \frac{n}{2}\right\rceil
+1
\right)
\left(
(-1)^{\lfloor n/2\rfloor}T_n(X)+2
\right)
\qquad \quad \hbox{for all $n\in \N$}.
\end{gather*}
Using Lemma~\ref{lem:Tchara}(ii) it follows that
\begin{gather*}
\Phi_n(0,0,0;X)=2
\left(
\left\lfloor \frac{n}{2}\right\rfloor
-
\left\lceil \frac{n}{2}\right\rceil
+1
\right)
\left(
(-1)^{\lfloor n/2\rfloor}T_{\lfloor n/2\rfloor}(X)+2
\right)
\qquad \quad \hbox{for all $n\in \N$}.
\end{gather*}
A similar way shows that
\begin{gather*}
\Psi_n(0,0,0;X)=T_n(X)
+
4\left(
\left\lfloor \frac{n}{2}\right\rfloor
-
\left\lceil \frac{n}{2}\right\rceil
+1
\right)
\left(
(-1)^{\lfloor n/2\rfloor}2\,T_{\lfloor n/2\rfloor}(X)+3
\right)
\qquad \quad \hbox{for all $n\in \N$}.
\end{gather*}

Applying the above results, the image of the relation in Theorem~\ref{thm:relZDelta} under $\iota$ implies the relation (\ref{e:presentationV}). By a similar proof to that of Theorem~\ref{thm:presentation}, the presentation for $Z(\V)$ conjectured by Iorgov \cite[Conjecture~1]{iog02} can be shown to be true.

\section{The center of DAHA of type $(C_1^\vee,C_1)$ at roots of unity}\label{s:ZH}

Fix four nonzero parameters $k_0$, $k_1$, $k_0^\vee$, $k_1^\vee$ taken from $\F$. The DAHA of type $(C_1^\vee,C_1)$ is an algebra $\H=\H_q(k_0,k_1,k_0^\vee,k_1^\vee)$ generated by $t_0$, $t_1$, $t_0^\vee$, $t_1^\vee$ subject to the relations
\begin{gather}
\begin{split}\label{e:h1}
(t_0-k_0)(t_0-k_0^{-1})=0,
\qquad \qquad
(t_1-k_1)(t_1-k_1^{-1})=0,\\
(t_0^\vee-k_0^\vee)(t_0^\vee-k_0^{\vee-1})=0,
\qquad \quad
(t_1^\vee-k_1^\vee)(t_1^\vee-k_1^{\vee-1})=0,
\end{split}
\\
t_0^\vee t_0 t_1^\vee t_1=q^{-1}.
\label{e:h2}
\end{gather}
The goal of \S\ref{s:ZH} is to give a presentation for $Z(\H)$.

\subsection{An $\N$-filtration structure of $\H$}
Let $c_0$, $c_1$, $c_0^\vee$, $c_1^\vee$ denote the scalars
\begin{gather*}
k_0+k_0^{-1},
\qquad
k_1+k_1^{-1},
\qquad
k_0^\vee+k_0^{\vee-1},
\qquad
k_1^\vee+k_1^{\vee-1}
\end{gather*}
respectively.
By (\ref{e:h1}) the generators $t_i$ and $t_i^\vee$ ($i=0,1$) are invertible with
$t_i^{-1}=c_i-t_i$ and
$t_i^{\vee-1}=c_i^\vee-t_i^\vee$.
Set
\begin{gather*}
u=t_1t_0^\vee,\qquad \quad
v=t_1^\vee t_1.
\end{gather*}
Clearly $t_0^\vee=t_1^{-1} u$ and $t_1^\vee=v t_1^{-1}$.
Relation (\ref{e:h2}) implies that $t_0=q^{-1} u^{-1} t_1 v^{-1}$.
Therefore $t_1^{\pm 1}$, $u^{\pm 1}$, $v^{\pm 1}$ form a set of generators of $\H$. Furthermore \cite[Proposition~5.2]{koo07} implies that

\begin{prop}\label{prop:presenH}
The $\F$-algebra $\H$ is generated by
$t_1^{\pm 1}$, $u^{\pm 1}$, $v^{\pm 1}$ subject to the relations
$t_1t_1^{-1}=t_1^{-1}t_1=1$, $uu^{-1}=u^{-1}u=1$, $vv^{-1}=v^{-1}v=1$ and
\begin{eqnarray*}
t_1^{-1}&=&
c_1-t_1,
\\
ut_1^{-1}&=&
-t_1u^{-1}+c_0^\vee,
\\
v t_1^{-1} &=&
-t_1 v^{-1}+c_1^\vee,
\\
v^{-1} u^{-1} &=&
q^2 u^{-1} v^{-1}
+q^2c_1t_1^{-1} u v^{-1}
-q^2c_1^\vee t_1^{-1} u
-q^2 c^\vee_0 t_1^{-1} v^{-1}
+qc_0t_1^{-1},
\\
vu &=&
q^2 uv
+q^{-2}c_1t_1^{-1} uv^{-1}
-c_1^\vee t_1^{-1} u
+c_0^\vee
(
t_1-q^{-2}c_1
) v^{-1}
-q c_0
(
t_1-q^{-2}c_1
),
\\
v^{-1}u &=&
-q^{-2}t_1^{-2}u v^{-1}
+c_1^\vee t_1^{-1} u
+q^{-2}c_0^\vee t_1^{-1} v^{-1}
-q^{-1}c_0t_1^{-1},
\\
v u^{-1}&=&
q^{-2}u^{-1}v
-q^{-2}c_1t_1^{-1} uv^{-1}
+q^{-2}c_1^\vee t_1^{-1} u
+c_0^\vee
(
q^{-2}c_1-t_1
)
v^{-1}
-q^{-1}c_0t_1^{-1}\\
&&\quad+\;\;(1-q^{-2})c_0^\vee c_1^\vee.
\end{eqnarray*}
\end{prop}

\noindent Applying the Bergman's diamond lemma \cite[Theorem~1.2]{berg1978} to Proposition~\ref{prop:presenH}, it follows that

\begin{lem}\label{lem:basisH}
$\H$ has the $\F$-basis
\begin{gather*}
t_1^\ell u^i v^j
\qquad \quad
\hbox{for all $\ell\in \{0,1\}$ and
$i,j\in \Z$}.
\end{gather*}
\end{lem}

\noindent Let $\T$ denote the $\F$-subalgebra of $\H$ generated by $t_1$.

\begin{lem}\label{lem:presenT}
$\T$ is the $\F$-algebra generated by $t_1^{\pm 1}$ subject to the relations
\begin{gather*}
t_1t_1^{-1}=t_1^{-1}t_1=1,
\qquad \quad
t_1+t_1^{-1}=c_1.
\end{gather*}
\end{lem}
\begin{proof}
Apply the first and fourth relations in Proposition~\ref{prop:presenH}.
\end{proof}

To make the arguments brief, Lemma~\ref{lem:presenT} will be used tacitly henceforward. For each $n\in \N$ denote by $\H_n$ the left $\T$-submodule of $\H$ spanned by
\begin{gather*}
u^i v^j
\qquad \quad
\hbox{for all $i$, $j\in \Z$ with $|i|+|j|\leq n$}.
\end{gather*}
As a left $\T$-module, $\H$ has the basis
\begin{gather*}
u^i v^j
\qquad \quad
\hbox{for all $i$, $j\in \Z$}
\end{gather*}
by Lemmas~\ref{lem:basisH}. Thus $\H=\bigcup_{n\in \N} \H_n$.  For any $m$, $n\in \N$, $\H_m\cdot\H_n\subseteq \H_{m+n}$ by Proposition~\ref{prop:presenH}.
Therefore the increasing sequence
\begin{gather*}
\H_0\subseteq \H_1\subseteq \cdots \subseteq \H_n\subseteq \cdots
\end{gather*}
gives an $\N$-filtration of the $\T$-algebra $\H$.
From the relations in Proposition~\ref{prop:presenH} we see that
\begin{eqnarray}
ut_1^{-1}&=&
-t_1u^{-1}
\qquad
\pmod{\H_0},
\label{e:ut-1}\\
v t_1^{-1} &=&
-t_1 v^{-1}
\qquad
\pmod{\H_0},
\label{e:vt-1}\\
v^{-1} u^{-1} &=&
q^2 u^{-1} v^{-1}
+q^2c_1t_1^{-1} u v^{-1}
\qquad
\pmod{\H_1},
\label{e:v-1u-1}
\\
v^{-1}u &=&
-q^{-2}t_1^{-2}u v^{-1}
\qquad
\pmod{\H_1},
\label{e:v-1u}\\
v u^{-1}&=&
q^{-2}u^{-1}v
-q^{-2}c_1t_1^{-1} uv^{-1}
\qquad
\pmod{\H_1}.
\label{e:vu-1}
\end{eqnarray}

\subsection{The centralizer of $t_1$ in $\H$}

By \cite[\S6]{koo07} or \cite[\S16]{aw&daha2013}
there exists a unique $\F$-algebra homomorphism $\sharp:\triangle\to \H$ given by
\begin{eqnarray*}
A^\sharp &=& t_1 t_0^\vee +(t_1t_0^\vee)^{-1}, \\
B^\sharp &=& t_1^\vee t_1+(t_1^\vee t_1)^{-1},\\
C^\sharp &=& t_0 t_1+(t_0 t_1)^{-1},\\
\alpha^\sharp &=&
\phi_1^{\ubar 0}
(c_0^\vee,c_1^\vee,c_0;q^{-1}t_1+qt_1^{-1}),
\\
\beta^\sharp &=&
\phi_1^{\ubar 1}
(c_0^\vee,c_1^\vee,c_0;q^{-1}t_1+qt_1^{-1}),
\\
\gamma^\sharp &=&
\phi_1^{\ubar 2}
(c_0^\vee,c_1^\vee,c_0;q^{-1}t_1+qt_1^{-1}),
\\
\Omega^\sharp &=&
\psi_1
(c_0^\vee,c_1^\vee,c_0;q^{-1}t_1+qt_1^{-1}).
\end{eqnarray*}
The image of $
\sharp$ is contained in the centralizer $C_\H(t_1)$ of $t_1$ in $\H$ \cite[Theorem 2.4(i)]{it10}.
For convenience we denote by $A$, $B$, $C$ the elements
\begin{gather*}
t_1 t_0^\vee+(t_1 t_0^\vee)^{-1},
\qquad \quad
t_1^\vee t_1+(t_1^\vee t_1)^{-1},
\qquad \quad
t_0 t_1+(t_0 t_1)^{-1}
\end{gather*}
respectively.  In this subsection we shall show that

\begin{thm}\label{thm:basisCT}
The elements
\begin{gather*}
A^i C^j B^k
\qquad \quad
\hbox{for all $i,j,k\in \N$ with $ijk=0$}
\end{gather*}
form a $\T$-basis of $C_\H(t_1)$.
\end{thm}

Before delving into the proof of Theorem~\ref{thm:basisCT} we lay some groundwork.  Applying (\ref{e:ut-1}), (\ref{e:vt-1}) the routine inductions yield that

\begin{lem}\label{lem:uit}
For each $i\in \N^*$
\begin{eqnarray*}
u^i t_1^{-1}
&=&
-t_1u^{-i}
\qquad \pmod{\H_{i-1}},
\label{e:uit}\\
v^i t_1^{-1}
&=&
-t_1 v^{-i}
\qquad \pmod{\H_{i-1}}.
\label{e:vit}
\end{eqnarray*}
\end{lem}

\noindent As consequences of Lemma~\ref{lem:uit} we have
\begin{lem}\label{lem:uivjt}
For all $i$, $j\in \N^*$
\begin{eqnarray*}
u^i v^j t_1^{-1}
&=&
-t_1 u^{-i} v^{-j}
-c_1 u^i v^{-j}
\qquad \pmod{\H_{i+j-1}},
\label{e:uivjt}
\\
u^{-i} v^j t_1^{-1}
&=&t_1^{-1} u^i v^{-j}
\qquad \pmod{\H_{i+j-1}},
\label{e:u-ivjt}
\\
u^{-i} v^{-j} t_1
&=&
-t_1^{-1} u^i v^j
-c_1 u^{-i} v^j
\qquad \pmod{\H_{i+j-1}}.
\label{e:u-iv-jt}
\end{eqnarray*}
\end{lem}

\begin{prop}\label{prop:ABC}
For each $i\in \N^*$
\begin{eqnarray*}
A^i&=&
u^i+u^{-i}
\qquad
\pmod{\H_{i-1}},
\\
B^i&=&
v^i+v^{-i}
\qquad
\pmod{\H_{i-1}},
\\
C^i&=&
(-1)^i q^{-i^2}
(u^{-i} v^i- t_1^{-2} u^i v^{-i})
\qquad
\pmod{\H_{2i-1}}.
\end{eqnarray*}
\end{prop}
\begin{proof}
By construction
\begin{gather*}
A=u+u^{-1},
\qquad \quad
B=v+v^{-1}.
\end{gather*}
The equalities for $A^i$, $B^i\bmod{\H_{i-1}}$ are immediate from the binomial theorem. A direct calculation yields that
\begin{gather*}
C=q^{-1}t_1^{-2} u v^{-1}
-q^{-1} u^{-1} v
-q^{-1} c_1^\vee t_1^{-1} u
-q^{-1} c_0^\vee t_1^{-1} v^{-1}
+c_0t_1^{-1}
+q^{-1} c_0^\vee c_1^\vee
\end{gather*}
as a linear combination of the basis of $\H$ given in Lemma~
\ref{lem:basisH}.
The expression implies that
\begin{gather*}
C^i=q^{-i}(r-s)^i
\qquad
\pmod{\H_{2i-1}}
\end{gather*}
where $r=t_1^{-2} u v^{-1}$ and $s= u^{-1} v$. Equation (\ref{e:v-1u-1}) implies that
\begin{gather*}
rs\in \H_3.
\end{gather*}
Decompose the product $sr$ into $(u^{-1} v t_1^{-1})(t_1^{-1} u v^{-1})$. The second equation in Lemma~\ref{lem:uivjt} with $(i,j)=(1,1)$ says that
\begin{gather}\label{e:u-1vt}
u^{-1}vt_1^{-1}=t_1^{-1}uv^{-1}
\qquad
\pmod{\H_1}.
\end{gather}
Therefore
\begin{gather*}
sr=t_1^{-1} u v^{-1} u^{-1} v t_1^{-1}
\qquad
\pmod{\H_3}.
\end{gather*}
By (\ref{e:v-1u-1}) the right-hand side is equal to $0\bmod{\H_3}$. Therefore
\begin{gather*}
sr\in \H_3.
\end{gather*}
Since $r$ and $s$ are in $\H_2$, the relations $rs$, $sr\in \H_3$ imply that $(r-s)^i=r^i+(-1)^i s^i\bmod{\H_{2i-1}}$. Applying (\ref{e:vu-1}) a routine induction shows that
\begin{gather*}
s^i=
q^{-i(i-1)} u^{-i} v^i
\qquad \pmod{\H_{2i-1}}.
\end{gather*}
To see the equality for $C^i\bmod{\H_{2i-1}}$ we have to show that
\begin{gather}
r^i =
(-1)^{i-1}q^{-i(i-1)}t_1^{-2} u^i v^{-i}
\qquad \pmod{\H_{2i-1}}.
\label{e:ri}
\end{gather}
Proceed by induction on $i$. Applying the induction hypothesis
\begin{gather*}
r^i=r\cdot r^{i-1}=(-1)^{i}q^{-(i-1)(i-2)}t_1^{-2} u v^{-1} t_1^{-2} u^{i-1} v^{1-i}
\qquad
\pmod{\H_{2i-1}}
\end{gather*}
for $i\geq 2$. Equation (\ref{e:u-1vt}) can be written as $uv^{-1}t_1^{-1}=c_1 u v^{-1}-t_1 u^{-1} v\bmod{\H_1}$. Therefore $r^i\bmod{\H_{2i-1}}$ is equal to
$(-1)^{i}q^{-(i-1)(i-2)}$ times
\begin{gather}\label{e:ri-1}
 c_1 t_1^{-2} u v^{-1} t_1^{-1} u^{i-1} v^{1-i}
 -t_1^{-1} u^{-1} v t_1^{-1} u^{i-1} v^{1-i}.
\end{gather}
Consider the element $w=u v^{-1} t_1^{-1} u$ in the minuend of (\ref{e:ri-1}). Equation (\ref{e:ut-1}) implies that $w=-u v^{-1} u^{-1} t_1\bmod{\H_2}$. Followed by using equation (\ref{e:v-1u}) it yields that
\begin{gather*}
w=q^2 t_1^2 v^{-1} t_1
\qquad
\pmod{\H_2}.
\end{gather*}
This shows $w\in \H_2$ and thus the minuend of (\ref{e:ri-1}) lies in $\H_{2i-1}$. Therefore $r^i \bmod{H_{2i-1}}$ is simplified to be
\begin{gather*}
(-1)^{i-1} q^{-(i-1)(i-2)} t_1^{-1} u^{-1} v t_1^{-1} u^{i-1} v^{1-i}.
\end{gather*}
By (\ref{e:u-1vt}) it follows that
\begin{gather*}
r^i=(-1)^{i-1} q^{-(i-1)(i-2)} t_1^{-2} u v^{-1} u^{i-1} v^{1-i}
\qquad \pmod{\H_{2i-1}}.
\end{gather*}
To see (\ref{e:ri}) it remains to prove that
\begin{gather*}
u v^{-1} u^{i-1}=q^{-2(i-1)} u^{i} v^{-1}
\qquad
\pmod{\H_{i}}.
\end{gather*}
The equation follows by applying the case $i=2$ to a routine induction.
By left multiplication by $u$ on (\ref{e:v-1u}) we have
\begin{gather*}
u v^{-1} u=
- q^{-2} u t_1^{-2} u v^{-1} u
\qquad
\pmod{\H_2}.
\end{gather*}
Applying (\ref{e:ut-1}) twice, one may find that
\begin{gather*}
u t_1^{-2}u=-u^2
\qquad
\pmod{\H_1}.
\end{gather*}
The case $i=2$ now follows by combing the above two equations. This proposition follows.
\end{proof}

It is ready to prove Theorem~\ref{thm:basisCT}.

\medskip

\noindent{\it Proof of Theorem~\ref{thm:basisCT}.}
Proceed by induction on $n\in \N$ to show that $C_\H(t_1)\cap \H_n$ has the $\T$-basis
\begin{gather*}
A^i C^j B^k
\qquad \quad
\hbox{for all $i,j,k\in \N$ with $i+2j+k\leq n$ and $ijk=0$}.
\end{gather*}
There is nothing to prove for $n=0$. Suppose that $n\geq 1$.
By induction hypothesis it suffices to show that for any $R\in C_\H(t_1)\cap \H_n$ the equation
\begin{gather}\label{e:R}
R=
\sum_{i, j\in \N
\atop
i+j=n}
s_{i,j} A^i B^j
+
\sum_{i\in \N, j\in \N^*
\atop
i+2j=n}
t_{i,j} A^i C^j
+
\sum_{i, j\in \N^*
\atop
i+2j=n}
u_{i,j} C^j B^i
\qquad
\pmod{\H_{n-1}}
\end{gather}
has a unique solution for $s_{i,j}$, $t_{i,j}$, $u_{i,j}\in \T$. By Lemma~\ref{lem:basisH} there are unique
$r_{i,j}\in \T$ for all $i,j\in \Z$ with $|i|+|j|=n$ such that
\begin{gather*}
R=\sum_{i,j\in \Z
\atop
|i|+|j|=n
}
r_{i,j}
u^i v^j
\qquad
\pmod{\H_{n-1}}.
\end{gather*}
Apply Lemmas~\ref{lem:uit} and \ref{lem:uivjt} to evaluate the coefficients of $u^i v^j$ in $[R,t_1]$ for all $i,j\in\Z$ with $|i|+|j|=n$. On the other hand, since $R\in C_\H(t_1)$ the commutator $[R,t_1]=0$ and it follows that
\begin{gather*}
r_{i,j}=r_{-i,-j},
\qquad \quad
c_1r_{i,j}=t_1r_{i,-j}+t_1^{-1}r_{-i,j}
\end{gather*}
for all $i$, $j\in \N$
with $i+j=n$.
As a result, $R \bmod{\H_{n-1}}$ is equal to
\begin{gather*}
r_{n,0}(u^n+u^{-n})
+
r_{0,n}(v^n+v^{-n})
+
\sum_{i,j\in \N^*
\atop i+j=n}
r_{i,j}(u^i+u^{-i})(v^j+v^{-j})
+
(r_{i,j}-r_{i,-j})
(t_1^2 u^{-i} v^j-u^i v^{-j}).
\end{gather*}
By Proposition~\ref{prop:ABC} we have
\begin{eqnarray*}
A^n &=& u^n+u^{-n}
\qquad \pmod{\H_{n-1}},\\
B^n &=& v^n+v^{-n}
\qquad \pmod{\H_{n-1}},
\end{eqnarray*}
and a direct calculation yields that
\begin{gather*}
A^i B^j= (u^i+u^{-i})(v^j+v^{-j})
\qquad \pmod{\H_{n-1}}
\end{gather*}
for all $i$, $j\in \N^*$ with $i+j=n$,
\begin{gather*}
A^i C^j =(-1)^j q^{-j^2}(u^{-i-j}v^j-t_1^{-2}u^{i+j} v^{-j})
\qquad \pmod{\H_{n-1}}
\end{gather*}
for all $i\in \N$, $j\in \N^*$ with $i+2j=n$ and
\begin{gather*}
C^j B^i =(-1)^j q^{-j^2}(u^{-j}v^{i+j}-t_1^{-2}u^j v^{-i-j})
\qquad \pmod{\H_{n-1}}
\end{gather*}
for all $i$, $j\in \N^*$ with $i+2j=n$. The comparison of coefficients implies that equation (\ref{e:R}) has the following unique solution
\begin{eqnarray*}
s_{i,j}&=&r_{i,j},
\\
t_{i,j}&=&
(-1)^j q^{j^2} t_1^2
(r_{i+j,j}-r_{i+j,-j}),
\\
u_{i,j} &=&
(-1)^j q^{j^2} t_1^2
(r_{j,i+j}-r_{j,-i-j}),
\end{eqnarray*}
as desired.
\hfill $\square$

\subsection{A presentation for $Z(\H)$}
Now the ideas from \S\ref{s:3centralABC} and \S\ref{s:ZDelta} can be used to give a presentation for $Z(\H)$.

\begin{thm}\label{thm:3centralH}
For any nonzero multiple $n$ of $\dbar$
the elements
\begin{gather*}
T_n(A) =
(t_1 t_0^\vee)^{n}+(t_1 t_0^\vee)^{-n},
\qquad
T_n(B) =
(t_1^\vee t_1)^{n}+(t_1^\vee t_1)^{-n},
\qquad
T_n(C) =
(t_0 t_1)^{n}+(t_0 t_1)^{-n}
\end{gather*}
are central in $\H$.
\end{thm}
\begin{proof}
By \cite[Proposition~2.4]{ob04} there exists an $\F$-algebra isomorphism from $\H=\H_q(k_0,k_1,k_0^\vee,k_1^\vee)$ into $\H_q(k_1^\vee,k_1,k_0,k_0^\vee)$ that sends
\begin{eqnarray*}
(t_0,t_1,t_0^\vee,t_1^\vee)
&\mapsto&
(t_0^\vee,t_1,t_1^{-1} t_1^\vee t_1,t_0),
\end{eqnarray*}
where $t_0$, $t_1$, $t_0^\vee$, $t_1^\vee$ on the right-hand side denote the defining generators of $\H_q(k_1^\vee,k_1,k_0,k_0^\vee)$ with the relations (\ref{e:h1}) in which $k_0$, $k_1$, $k_0^\vee$, $k_1^\vee$ are replaced by $k_1^\vee$, $k_1$, $k_0$, $k_0^\vee$ respectively.
A direct calculation shows that the isomorphism sends $A$, $B$, $C\in \H$ to the corresponding $B$, $C$, $A \in \H_q(k_1^\vee,k_1,k_0,k_0^\vee)$ respectively. Since $k_0$, $k_1$, $k_0^\vee$, $k_1^\vee$ are arbitrary nonzero scalars in $\F$, it suffices to show that $T_n(B)$ is central in $\H$ when $\dbar$ divides $n$.

As mentioned before, the element $B\in C_\H(t_1)$. Clearly $B=v+v^{-1}$ commutes with $v$. Since the $\F$-algebra $\H$ is generated by $t_1^{\pm 1}$, $u^{\pm 1}$, $v^{\pm 1}$ by Proposition~\ref{prop:presenH}, it is enough to verify that $T_n(B)=v^n+v^{-n}$ commutes with $u$ when $\dbar$ divides $n$.
To see this express $B^nu$ ($n\in \N$) as a linear combination of the $\T$-basis of $\H$  given below Lemma~\ref{lem:presenT}:
\begin{gather*}
B^nu=u(q^2 v+q^{-2} v^{-1})^n
+t_1(q^{-1}c_0^\vee v^{-1}-c_0)
\,
\frac
{(q^2 v+q^{-2} v^{-1})^n-(v+v^{-1})^n}
{q v-q^{-1}v^{-1}}.
\end{gather*}
Using this identity the element $T_n(B)u$ is equal to
\begin{gather*}
u
T_n(q^2 v+q^{-2} v^{-1})
+
t_1(q^{-1}c_0^\vee v^{-1}-c_0)
\,
\frac
{T_n(q^2 v+q^{-2} v^{-1})-T_n(v+v^{-1})}
{q v-q^{-1}v^{-1}}.
\end{gather*}
Simplifying the above by using Lemma~\ref{lem:Tchara}(i) it becomes
\begin{gather*}
u(q^{2n} v^n+q^{-2n} v^{-n})
+t_1(q^n-q^{-n})
(q^{-1}c_0^\vee v^{-1}-c_0)
\,
\frac
{q^nv^n-q^{-n}v^{-n}}
{q v-q^{-1}v^{-1}}.
\end{gather*}
Since $q^{2\dbar}=1$ the above is equal to $u T_n(B)$ when $\dbar$ divides $n$, as desired.
\end{proof}

\begin{thm}\label{thm:basisZH}
The elements
\begin{gather*}
T_{\dbar}(A)^i
T_{\dbar}(B)^j
T_{\dbar}(C)^k
\qquad \quad
\hbox{for all $i,j,k\in \N$ with $ijk=0$}
\end{gather*}
form an $\F$-basis of $Z(\H)$.
\end{thm}
\begin{proof}
Inspired by Theorem~\ref{thm:basisCT} we consider the $\N$-filtration
\begin{gather*}
C_\H(t_1)_0\subseteq
C_\H(t_1)_1\subseteq
\cdots
\subseteq
C_\H(t_1)_n
\subseteq
\cdots
\end{gather*}
of the $\T$-algebra $C_\H(t_1)$, where $C_\H(t_1)_n$ is defined to be the $\T$-submodule of $C_\H(t_1)$ spanned by
\begin{gather*}
A^i C^j B^k
\qquad \quad
\hbox{for all $i$, $j$, $k\in \N$ with $ijk=0$ and $i+j+k\leq n$}
\end{gather*}
rather than $C_\H(t_1)\cap \H_n$.  Applying the $\N$-filtration of $C_\H(t_1)$ and by Theorems~\ref{thm:basisCT} and \ref{thm:3centralH}, a similar argument to the proof of Theorem~\ref{thm:basisZDelta} show that $Z(\H)$ has the $\T$-basis
\begin{gather*}
T_{\dbar}(A)^i
T_{\dbar}(B)^j
T_{\dbar}(C)^k
\qquad \quad
\hbox{for all $i$, $j$, $k\in \N$ with $ijk=0$.}
\end{gather*}
Therefore any element $R\in Z(\H)$ can be uniquely expressed as a linear combination
\begin{gather*}
\sum
_{i,j,k\in \N
\atop
ijk=0}
r_{i,j,k}\,T_{\dbar}(A)^i T_{\dbar}(B)^j T_{\dbar}(C)^k
\end{gather*}
where the coefficients $r_{i,j,k}\in \T$. It remains to prove that the coefficients $r_{i,j,k}\in \F$. To do this we invoke the fact:

\begin{lem}\label{lem:TcapZH}
$\T\cap Z(\H)=\F$.
\end{lem}
\begin{proof}
By (\ref{e:ut-1}) the commutator $[t_1^{-1},u]=t_1^{-1}u+t_1 u^{-1}\bmod{\H_0}$. Therefore $[t_1^{-1},u]\not=0 \bmod{\H_0}$ by Lemma~\ref{lem:basisH}. This shows that $t_1^{-1}\not\in Z(\H)$.  Since $\{1,t_1^{-1}\}$ is an $\F$-basis of $\T$ this lemma follows.
\end{proof}

\noindent Since $R$ is central in $\H$ the commutator $[R,S]=0$ for any $S\in \H$. On the other hand, Theorem~\ref{thm:3centralH} implies that $[R,S]$ is equal to
\begin{gather*}
\sum
_{i,j,k\in \N
\atop
ijk=0}
[r_{i,j,k},S]\,T_{\dbar}(A)^i T_{\dbar}(B)^j T_{\dbar}(C)^k.
\end{gather*}
This forces that $[r_{i,j,k},S]=0$ and thus $r_{i,j,k}\in Z(\H)$. By Lemma~\ref{lem:TcapZH} the coefficients $r_{i,j,k}\in \F$ as claimed.
\end{proof}

The Oblomkov presentation for $Z(\H)$ at $q=1$ \cite[Theorem~3.1]{ob04} now can be generalized to any root of unity $q$.

\begin{thm}\label{thm:presenZH}
Let
\begin{gather*}
A= t_1 t_0^\vee+(t_1t_0^\vee)^{-1},
\qquad
B=t_1^\vee t_1+(t_1^\vee t_1)^{-1},
\qquad
C=t_0t_1+(t_0t_1)^{-1}.
\end{gather*}
Let $c_0$, $\tilde c_1$, $c_0^\vee$, $c_1^\vee$ denote the scalars
\begin{gather*}
k_0+k_0^{-1},
\qquad
q^{-1}k_1+q k_1^{-1},
\qquad
k_0^\vee+k_0^{\vee-1},
\qquad
k_1^\vee+k_1^{\vee-1}
\end{gather*}
respectively.
Then $Z(\H)$ is the commutative $\F$-algebra generated by $T_{\dbar}(A)$, $T_{\dbar}(B)$, $T_{\dbar}(C)$ subject to the relation
\begin{align}\label{e:relZH}
\begin{split}
&
q^{\dbar}  \phi_{\dbar}^{\ubar 0}(c_0^\vee,c_1^\vee,c_0; \tilde c_1) T_{\dbar}(A)
+q^{\dbar} \phi_{\dbar}^{\ubar 1}(c_0^\vee,c_1^\vee,c_0;\tilde c_1) T_{\dbar}(B)
+q^{\dbar} \phi_{\dbar}^{\ubar 2}(c_0^\vee,c_1^\vee,c_0;\tilde c_1) T_{\dbar}(C)
\\
&\quad=\;\;
q^{\dbar} T_{\dbar}(A) T_{\dbar}(B) T_{\dbar}(C)
+T_{\dbar}(A)^2
+ T_{\dbar}(B)^2
+ T_{\dbar}(C)^2
+\psi_{\dbar}(c_0^\vee,c_1^\vee,c_0;\tilde c_1)-2.
\end{split}
\end{align}
\end{thm}
\begin{proof}
As a consequence of Theorem~\ref{thm:basisZH} the center of $\H$ is generated by $T_{\dbar}(A)$, $T_{\dbar}(B)$, $T_{\dbar}(C)$.
By Lemma~\ref{lem:Tchara}(i), for each $n\in \N$
\begin{gather*}
T_n(q^{-1}t_1+q t_1^{-1})=
q^{-n}t_1^n+q^nt_1^{-n},
\qquad \quad
T_n(\tilde c_1)=
q^{-n}k_1^n+q^nk_1^{-n}.
\end{gather*}
Both are equal to $q^n T_n(c_1)$ when $\dbar$ divides $n$. Recall the homomorphism $\sharp:\triangle\to \H$ from Theorem~\ref{thm:basisCT} above. Applying Theorem~\ref{thm:poly} it follows that
\begin{gather*}
\Phi_{\dbar}(\alpha,\beta,\gamma;\Omega)^\sharp
=\phi_{\dbar}(c_0^\vee,c_1^\vee,c_0;\tilde c_1),
\qquad \quad
\Psi_{\dbar}(\alpha,\beta,\gamma;\Omega)^\sharp
=\psi_{\dbar}(c_0^\vee,c_1^\vee,c_0;\tilde c_1).
\end{gather*}
Therefore
(\ref{e:relZH}) holds by Theorem~\ref{thm:relZDelta}.

Let $K$ denote the $\F$-subalgebra of $Z(\H)$ generated by $T_{\dbar}(A)$, $T_{\dbar}(B)$. Relation (\ref{e:relZH}) implies that $T_{\dbar}(C)$ is a root of the quadratic polynomial
\begin{align*}
P(X)
&=
X^2
+q^{\dbar}
(
T_{\dbar}(A) T_{\dbar}(B)
-\phi_{\dbar}^{\ubar 2}(c_0^\vee,c_1^\vee,c_0;\tilde c_1)
)X
+T_{\dbar}(A)^2
+T_{\dbar}(B)^2
\\
&\qquad-\;
q^{\dbar}\phi_{\dbar}^{\ubar 0}(c_0^\vee,c_1^\vee,c_0; \tilde c_1) T_{\dbar}(A)
-q^{\dbar}\phi_{\dbar}^{\ubar 1}(c_0^\vee,c_1^\vee,c_0; \tilde c_1) T_{\dbar}(B)
+
\psi_{\dbar}(c_0^\vee,c_1^\vee,c_0;\tilde c_1)
-2.
\end{align*}
By Theorem~\ref{thm:basisZH}, $K$ is a two-variable polynomial ring over $\F$ generated by $T_{\dbar}(A)$, $T_{\dbar}(B)$. In particular $K$ is a unique factorization domain. Observe that $P(X)$ is irreducible over $K$. By Gauss's Lemma $P(X)$ is the minimal polynomial of $T_{\dbar}(C)$ over the fraction field of $K$. This shows the theorem.
\end{proof}

\appendix

\setcounter{secnumdepth}{0}
\section{Appendix. Auxiliary data
}\label{a:homog}

Assume that $q^2\not=1$.
As the linear combinations of the $\F[a^{\pm1},b^{\pm 1},c^{\pm 1}]\otimes_\F \F[\Lambda]$-basis
(\ref{e:basis}) of $ \F[a^{\pm1},b^{\pm 1},c^{\pm 1}]\otimes_\F Z(U)$, the nonzero coefficients of
\begin{gather*}
T_{\dbar}(A)^\natural, \qquad
T_{\dbar}(B)^\natural, \qquad
T_{\dbar}(C)^\natural, \qquad
 T_{\dbar}(A)^{2\natural}, \qquad
T_{\dbar}(B)^{2\natural}, \qquad
T_{\dbar}(C)^{2\natural},\\
T_{\dbar}(A)^\natural T_{\dbar}(B)^\natural, \qquad
T_{\dbar}(B)^\natural T_{\dbar}(C)^\natural, \qquad
T_{\dbar}(C)^\natural T_{\dbar}(A)^\natural, \qquad
T_{\dbar}(A)^\natural T_{\dbar}(B)^\natural T_{\dbar}(C)^\natural
\end{gather*}
can be evaluated by applying relation (\ref{e:TLambda}) to Theorem~\ref{thm:embedd}. The results are listed as follows:

\begin{table}[H]
\begin{minipage}[H]{0.5\textwidth}
\centering
\extrarowheight=3.3pt
\begin{tabular}{c|l|c}
\multicolumn{2}{c}{coefficients}
\vline
&$T_{\dbar}(A)^\natural$\\

\midrule[2pt]

\multirow{2}{*}{$k^{\dbar i}$}
&$i=-1$
&$
a^{\dbar}
$\\
\cline{2-3}

&$i=1$
&$
a^{-\dbar}
$\\

\midrule[1pt]

\multirow{2}{*}{$k^{\dbar i}e^{\dbar}$}
&$i=-1$
&$
-(q-q^{-1})^{\dbar}
q^{\dbar} a^{\dbar}
$\\
\cline{2-3}

&$i=0$
&$
(q-q^{-1})^{\dbar} b^{\dbar} c^{-\dbar}
$
\end{tabular}
\end{minipage}%
f\begin{minipage}[H]{0.5\textwidth}
\centering
\extrarowheight=3.3pt
\begin{tabular}{c|l|c}
\multicolumn{2}{c}{coefficients}
\vline
&$T_{\dbar}(B)^\natural$\\

\midrule[2pt]

\multirow{2}{*}{$k^{\dbar i}f^{\dbar}$}
&$i=0$
&$
(q-q^{-1})^{\dbar} b^{-\dbar}
$\\
\cline{2-3}

&$i=1$
&$
-(q-q^{-1})^{\dbar}  q^{\dbar} a^{-\dbar} c^{\dbar}
$\\

\midrule[1pt]

\multirow{2}{*}{$k^{\dbar i}$}
&$i=-1$
&$
b^{-\dbar}
$\\
\cline{2-3}

&$i=1$
&$
b^{\dbar}
$
\end{tabular}
\end{minipage}
\end{table}

\begin{table}[H]
\centering
\extrarowheight=3.3pt
\begin{tabular}{c|l|c}
\multicolumn{2}{c}{coefficients}
\vline
&$T_{\dbar}(C)^\natural$\\

\midrule[2pt]

\multirow{2}{*}{$k^{\dbar i}f^{\dbar}$}
&$i=-1$
&$
-(q-q^{-1})^{\dbar} q^{\dbar} a^{\dbar} b^{-\dbar}
$\\
\cline{2-3}

&$i=0$
&$
(q-q^{-1})^{\dbar} c^{\dbar}
$\\

\midrule[1pt]

\multirow{2}{*}{$k^{\dbar i}$}
&$i=-2$
&$
-2 q^{\dbar} a^{\dbar} b^{-\dbar}
$\\
\cline{2-3}

&$i=-1$
&$
a^{\dbar} b^{-\dbar} T_{\dbar}(\Lambda)+c^{\dbar}+c^{-\dbar}
$\\

\midrule[1pt]

\multirow{2}{*}{$k^{\dbar i}e^{\dbar}$}

&$i=-2$
&$
(q-q^{-1})^{\dbar} a^{\dbar} b^{-\dbar}
$\\
\cline{2-3}

&$i=-1$
&$
-(q-q^{-1})^{\dbar} q^{\dbar} c^{-\dbar}
$
\end{tabular}
\end{table}

\begin{table}[H]
\begin{minipage}[H]{0.5\textwidth}
\centering
\extrarowheight=3.3pt
\begin{tabular}{c|l|c}
\multicolumn{2}{c}{coefficients} \vline &$T_{\dbar}(A)^{2\natural}$\\

\midrule[2pt]

\multirow{3}{*}{$k^{\dbar i}$}
&$i=-2$
&$
a^{2\dbar}
$\\
\cline{2-3}

&$i=0$
&$
2
$\\
\cline{2-3}

&$i=2$
&$
a^{-2\dbar}
$\\

\midrule[1pt]

\multirow{4}{*}{$k^{\dbar i}e^{\dbar}$}

&$i=-2$
&$
-2(q-q^{-1})^{\dbar} q^{\dbar} a^{2\dbar}
$\\
\cline{2-3}

&$i=-1$
&$
2(q-q^{-1})^{\dbar} a^{\dbar} b^{\dbar} c^{-\dbar}
$\\
\cline{2-3}

&$i=0$
&$
-2 (q-q^{-1})^{\dbar} q^{\dbar}
$\\
\cline{2-3}

&$i=1$
&$
2 (q-q^{-1})^{\dbar} a^{-\dbar} b^{\dbar} c^{-\dbar}
$\\

\midrule[1pt]

\multirow{3}{*}{$k^{\dbar i} e^{2\dbar}$}

&$i=-2$
&$
(q-q^{-1})^{2\dbar} a^{2\dbar}
$\\
\cline{2-3}

&$i=-1$
&$
-2(q-q^{-1})^{2\dbar} q^{\dbar} a^{\dbar} b^{\dbar} c^{-\dbar}
$\\
\cline{2-3}

&$i=0$
&$
(q-q^{-1})^{2\dbar}b^{2\dbar}c^{-2\dbar}
$
\end{tabular}
\end{minipage}%
\begin{minipage}[H]{0.5\textwidth}
\centering
\extrarowheight=3.3pt
\begin{tabular}{c|l|c}
\multicolumn{2}{c}{coefficients} \vline
&$T_{\dbar}(B)^{2\natural}$\\

\midrule[2pt]

\multirow{3}{*}{$k^{\dbar i}f^{2\dbar}$}

&$i=0$
&$
(q-q^{-1})^{2\dbar}
b^{-2\dbar}
$\\
\cline{2-3}

&$i=1$
&$
-2(q-q^{-1})^{2\dbar}q^{\dbar} a^{-\dbar}b^{-\dbar} c^{\dbar}
$\\
\cline{2-3}

&$i=2$
&$
(q-q^{-1})^{2\dbar}
a^{-2\dbar} c^{2\dbar}
$\\

\midrule[1pt]

\multirow{4}{*}{$k^{\dbar i}f^{\dbar}$}
&$i=-1$
&$
2(q-q^{-1})^{\dbar} b^{-2\dbar}
$\\
\cline{2-3}

&$i=0$
&$
-2(q-q^{-1})^{\dbar}
q^{\dbar} a^{-\dbar}b^{-\dbar} c^{\dbar}
$\\
\cline{2-3}

&$i=1$
&$
2(q-q^{-1})^{\dbar}
$\\
\cline{2-3}

&$i=2$
&$
-2 (q-q^{-1})^{\dbar} q^{\dbar} a^{-\dbar}b^{\dbar} c^{\dbar}
$\\

\midrule[1pt]

\multirow{3}{*}{$k^{\dbar i}$}
&$i=-2$
&$
b^{-2\dbar}
$\\
\cline{2-3}

&$i=0$
&$
2
$\\
\cline{2-3}

&$i=2$
&$
b^{2\dbar}
$
\end{tabular}
\end{minipage}
\end{table}

\begin{table}[H]
\centering
\extrarowheight=3.3pt
\begin{tabular}{c|l|c}
\multicolumn{2}{c}{coefficients}
\vline
&$T_{\dbar}(C)^{2\natural}$\\

\midrule[2pt]

\multirow{3}{*}{$k^{\dbar i}f^{2\dbar}$}
&$i=-2$
&$
(q-q^{-1})^{2\dbar} a^{2\dbar} b^{-2\dbar}
$\\
\cline{2-3}

&$i=-1$
&$
-2(q-q^{-1})^{2\dbar}
q^{\dbar} a^{\dbar} b^{-\dbar} c^{\dbar}
$\\
\cline{2-3}

&$i=0$
&$
(q-q^{-1})^{2\dbar}c^{2\dbar}
$\\

\midrule[1pt]

\multirow{3}{*}{$k^{\dbar i}f^{\dbar}$}
&$i=-3$
&$
4(q-q^{-1})^{\dbar}
a^{2\dbar}b^{-2\dbar}
$\\
\cline{2-3}

&$i=-2$
&$
-2 (q-q^{-1})^{\dbar} q^{\dbar} a^{\dbar} b^{-\dbar}
\left(
a^{\dbar} b^{-\dbar}T_{\dbar}(\Lambda)
+
3c^{\dbar}+c^{-\dbar}
\right)
$\\
\cline{2-3}

&$i=-1$
&$
2(q-q^{-1})^{\dbar} c^{\dbar}
\left(
a^{\dbar} b^{-\dbar}  T_{\dbar}(\Lambda)
+c^{\dbar}+c^{-\dbar}
\right)
$\\

\midrule[1pt]

\multirow{5}{*}{$k^{\dbar i}$}
&$i=-4$
&$
6a^{2\dbar}b^{-2\dbar}
$\\
\cline{2-3}

&$i=-3$
&$
-6q^{\dbar} a^{\dbar} b^{-\dbar}
\left(
a^{\dbar} b^{-\dbar}T_{\dbar}(\Lambda)
+c^{\dbar}+c^{-\dbar}
\right)
$\\
\cline{2-3}

&$i=-2$
&$

a^{2\dbar} b^{-2\dbar}\left(
T_{\dbar}(\Lambda)^2
+2\right)
+(c^{\dbar}+c^{-\dbar})
\left(
4 a^{\dbar} b^{-\dbar}T_{\dbar}(\Lambda)
+c^{\dbar}+c^{-\dbar}
\right)
+2
$\\
\cline{2-3}

&$i=-1$
&$-2
q^{\dbar} a^{\dbar} b^{-\dbar}
\left(
a^{-\dbar} b^{\dbar} T_{\dbar}(\Lambda)
+c^{\dbar}+c^{-\dbar}
\right)
$\\
\cline{2-3}

&$i=0$
&$
2
$
\\

\midrule[1pt]

\multirow{3}{*}{$k^{\dbar i}e^{\dbar}$}
&$i=-4$
&$
-4(q-q^{-1})^{\dbar} q^{\dbar} a^{2\dbar}b^{-2\dbar}
$\\
\cline{2-3}

&$i=-3$
&$
2(q-q^{-1})^{\dbar} a^{\dbar} b^{-\dbar}
\left(
a^{\dbar} b^{-\dbar} T_{\dbar}(\Lambda)
+
c^{\dbar}+3c^{-\dbar}
\right)
$\\
\cline{2-3}

&$i=-2$
&$
-2(q-q^{-1})^{\dbar} q^{\dbar} c^{-\dbar}
\left(
a^{\dbar} b^{-\dbar} T_{\dbar}(\Lambda)
+c^{\dbar}+c^{-\dbar}
\right)
$\\

\midrule[1pt]

\multirow{3}{*}{$k^{\dbar i}e^{2\dbar}$}
&$i=-4$
&$
(q-q^{-1})^{2\dbar}a^{2\dbar} b^{-2\dbar}
$\\
\cline{2-3}

&$i=-3$
&$
-2(q-q^{-1})^{2\dbar}q^{\dbar} a^{\dbar} b^{-\dbar} c^{-\dbar}
$\\
\cline{2-3}

&$i=-2$
&$
(q-q^{-1})^{2\dbar} c^{-2\dbar}
$\\
\end{tabular}
\end{table}

\begin{table}[H]
\centering
\extrarowheight=3.3pt
\begin{tabular}{c|l|c}
\multicolumn{2}{c}{coefficients}
\vline
&$T_{\dbar}(A)^\natural T_{\dbar} (B)^\natural$\\

\midrule[2pt]

\multirow{4}{*}{$k^{\dbar i}f^{\dbar}$}
&$i=-1$
&$
(q-q^{-1})^{\dbar} a^{\dbar} b^{-\dbar}
$\\
\cline{2-3}

&$i=0$
&$
-(q-q^{-1})^{\dbar} q^{\dbar} c^{\dbar}
$\\
\cline{2-3}

&$i=1$
&$
(q-q^{-1})^{\dbar} a^{-\dbar} b^{-\dbar}
$\\
\cline{2-3}

&$i=2$
&$
-(q-q^{-1})^{\dbar} q^{\dbar}
a^{-2\dbar} c^{\dbar}
$
\\

\midrule[1pt]

\multirow{5}{*}{$k^{\dbar i}$}
&$i=-2$
&$
2 a^{\dbar} b^{-\dbar}
$\\
\cline{2-3}

&$i=-1$
&$
-q^{\dbar}
\left(
a^{\dbar} b^{-\dbar} T_{\dbar}(\Lambda)
+
c^{\dbar}+c^{-\dbar}
\right)
$\\
\cline{2-3}

&$i=0$
&$
(c^{\dbar}+c^{-\dbar})T_{\dbar}(\Lambda)
+(a^{\dbar}+a^{-\dbar})(b^{\dbar}+b^{-\dbar})
$\\
\cline{2-3}

&$i=1$
&$
-q^{\dbar}
\left(
a^{-\dbar} b^{\dbar} T_{\dbar}(\Lambda)
+
c^{\dbar}+c^{-\dbar}
\right)
$\\
\cline{2-3}

&$i=2$
&$
2a^{-\dbar}b^{\dbar}
$\\

\midrule[1pt]

\multirow{4}{*}{$k^{\dbar i}e^{\dbar}$}
&$i=-2$
&$
-(q-q^{-1})^{\dbar} q^{\dbar} a^{\dbar} b^{-\dbar}
$\\
\cline{2-3}

&$i=-1$
&$
(q-q^{-1})^{\dbar}  c^{-\dbar}
$\\
\cline{2-3}

&$i=0$
&$
-(q-q^{-1})^{\dbar}  q^{\dbar} a^{\dbar} b^{\dbar}
$\\
\cline{2-3}

&$i=1$
&$
(q-q^{-1})^{\dbar}
b^{2\dbar}c^{-\dbar}
$\\
\end{tabular}
\end{table}

\begin{table}[H]
\centering
\extrarowheight=3.3pt
\begin{tabular}{c|l|c}
\multicolumn{2}{c}{coefficients}
\vline
&$T_{\dbar}(B)^\natural T_{\dbar} (C)^\natural$\\

\midrule[2pt]

\multirow{3}{*}{$k^{\dbar i}f^{2\dbar}$}
&$i=-1$
&$
-(q-q^{-1})^{2\dbar} q^{\dbar} a^{\dbar} b^{-2\dbar}
$\\
\cline{2-3}

&$i=0$
&$
2(q-q^{-1})^{2\dbar} b^{-\dbar} c^{\dbar}
$\\
\cline{2-3}

&$i=1$
&$
-(q-q^{-1})^{2\dbar} q^{\dbar} a^{-\dbar} c^{2\dbar}
$\\

\midrule[1pt]

\multirow{4}{*}{$k^{\dbar i}f^{\dbar}$}
&$i=-2$
&$
-3(q-q^{-1})^{\dbar}
q^{\dbar} a^{\dbar} b^{-2\dbar}
$\\
\cline{2-3}

&$i=-1$
&$
(q-q^{-1})^{\dbar} b^{-\dbar}
\left(
a^{\dbar} b^{-\dbar}T_{\dbar}(\Lambda)
+4c^{\dbar}+c^{-\dbar}
\right)
$\\
\cline{2-3}

&$i=0$
&$
-(q-q^{-1})^{\dbar}
q^{\dbar}
a^{-\dbar}
\left(
a^{\dbar} b^{-\dbar}c^{\dbar}
T_{\dbar}(\Lambda)
+
a^{2\dbar}+1+c^{2\dbar}
\right)
$\\
\cline{2-3}

&$i=1$
&$
(q-q^{-1})^{\dbar} b^{\dbar} c^{\dbar}
$\\

\midrule[1pt]

\multirow{5}{*}{$k^{\dbar i}$}
&$i=-3$
&$
-3 q^{\dbar} a^{\dbar} b^{-2\dbar}
$\\
\cline{2-3}

&$i=-2$
&$
2 b^{-\dbar}
\left(
a^{\dbar} b^{-\dbar} T_{\dbar}(\Lambda)
+c^{\dbar}+c^{-\dbar}
\right)
$\\
\cline{2-3}

&$i=-1$
&$
-q^{\dbar}
a^{\dbar}
\left(
a^{-\dbar} b^{-\dbar}
(c^{\dbar}+c^{-\dbar}) T_{\dbar}(\Lambda)
+
a^{-2\dbar}+2
+b^{-2\dbar}
\right)
$\\
\cline{2-3}

&$i=0$
&$
(a^{\dbar}+a^{-\dbar})T_{\dbar}(\Lambda)
+
(b^{\dbar}+b^{-\dbar})(c^{\dbar}+c^{-\dbar})
$\\
\cline{2-3}

&$i=1$
&$
-q^{\dbar} a^{-\dbar}
$
\\

\midrule[1pt]

\multirow{4}{*}{$k^{\dbar i}e^{\dbar}$}
&$i=-3$
&$
(q-q^{-1})^{\dbar} a^{\dbar} b^{-2\dbar}
$\\
\cline{2-3}

&$i=-2$
&$
-(q-q^{-1})^{\dbar} q^{\dbar} b^{-\dbar}c^{-\dbar}
$\\
\cline{2-3}

&$i=-1$
&$
(q-q^{-1})^{\dbar} a^{\dbar}
$\\
\cline{2-3}

&$i=0$
&$
-(q-q^{-1})^{\dbar} q^{\dbar} b^{\dbar} c^{-\dbar}
$
\end{tabular}
\end{table}

\begin{table}[H]
\centering
\extrarowheight=3.3pt
\begin{tabular}{c|l|c}
\multicolumn{2}{c}{coefficients}
\vline
&$T_{\dbar}(C)^\natural T_{\dbar}(A)^\natural$\\

\midrule[2pt]

\multirow{4}{*}{$k^{\dbar i}f^{\dbar}$}
&$i=-2$
&$
-(q-q^{-1})^{\dbar} q^{\dbar} a^{2\dbar} b^{-\dbar}
$\\
\cline{2-3}

&$i=-1$
&$
(q-q^{-1})^{\dbar} a^{\dbar} c^{\dbar}
$\\
\cline{2-3}

&$i=0$
&$
-(q-q^{-1})^{\dbar} q^{\dbar} b^{-\dbar}
$\\
\cline{2-3}

&$i=1$
&$
(q-q^{-1})^{\dbar} a^{-\dbar} c^{\dbar}
$\\

\midrule[1pt]

\multirow{5}{*}{$k^{\dbar i}$}
&$i=-3$
&$
-3 q^{\dbar} a^{2\dbar} b^{-\dbar}
$\\
\cline{2-3}

&$i=-2$
&$
2 a^{\dbar}
\left(
a^{\dbar} b^{-\dbar}T_{\dbar}(\Lambda)
+c^{\dbar}+c^{-\dbar}
\right)
$\\
\cline{2-3}

&$i=-1$
&$
-q^{\dbar}
b^{-\dbar}
\left(
a^{\dbar} b^{\dbar}(c+c^{-\dbar})
T_{\dbar}(\Lambda)
+a^{2\dbar}+2+b^{2\dbar}
\right)
$\\
\cline{2-3}

&$i=0$
&$
(b+b^{-\dbar})T_{\dbar}(\Lambda)
+(a^{\dbar}+a^{-\dbar})(c^{\dbar}+c^{-\dbar})
$\\
\cline{2-3}

&$i=1$
&$
-q^{\dbar} b^{\dbar}
$\\

\midrule[1pt]

\multirow{4}{*}{$k^{\dbar i}e^{\dbar}$}
&$i=-3$
&$
3 (q-q^{-1})^{\dbar} a^{2\dbar}b^{-\dbar}
$\\
\cline{2-3}

&$i=-2$
&$
-(q-q^{-1})^{\dbar} q^{\dbar} a^{\dbar}
\left(
a^{\dbar} b^{-\dbar}T_{\dbar}(\Lambda)
+c^{\dbar}+4c^{-\dbar}
\right)
$\\
\cline{2-3}

&$i=-1$
&$
(q-q^{-1})^{\dbar}
b^{\dbar}
\left(
a^{\dbar} b^{-\dbar} c^{-\dbar}T_{\dbar}(\Lambda)
+
b^{-2\dbar}+1+c^{-2\dbar}
\right)
$\\
\cline{2-3}

&$i=0$
&$
-(q-q^{-1})^{\dbar}
q^{\dbar} a^{-\dbar}c^{-\dbar}
$\\

\midrule[1pt]

\multirow{3}{*}{$k^{\dbar i}e^{2\dbar}$}
&$i=-3$
&$
-(q-q^{-1})^{2\dbar}
q^{\dbar} a^{2\dbar} b^{-\dbar}
$\\
\cline{2-3}

&$i=-2$
&$
2(q-q^{-1})^{2\dbar} a^{\dbar} c^{-\dbar}
$\\
\cline{2-3}

&$i=-1$
&$-(q-q^{-1})^{2\dbar}
q^{\dbar} b^{\dbar} c^{-2\dbar}
$
\end{tabular}
\end{table}

\begin{table}[H]
\centering
\extrarowheight=3.3pt
\begin{tabular}{c|l|c}
\multicolumn{2}{c}{coefficients} \vline
&$T_{\dbar}(A)^\natural T_{\dbar}(B)^\natural T_{\dbar}(C)^\natural$\\

\midrule[2pt]

\multirow{5}{*}{$k^{\dbar i} f^{2\dbar}$}
& $i=-2$
&$
-(q-q^{-1})^{2\dbar}q^{\dbar} a^{2\dbar} b^{-2\dbar}
$\\
\cline{2-3}

& $i=-1$
&$
2(q-q^{-1})^{2\dbar} a^{\dbar} b^{-\dbar} c^{\dbar}
$\\
\cline{2-3}

&$i=0$
&$
-(q-q^{-1})^{2\dbar}q^{\dbar}(b^{-2\dbar}+c^{2\dbar})
$\\
\cline{2-3}

& $i=1$
&$
2(q-q^{-1})^{2\dbar} a^{-\dbar} b^{-\dbar} c^{\dbar}
$\\
\cline{2-3}

& $i=2$
&$
-(q-q^{-1})^{2\dbar}q^{\dbar} a^{-2\dbar} c^{2\dbar}
$\\

\midrule[1pt]

\multirow{7}{*}{$k^{\dbar i} f^{\dbar}$}
& $i=-3$
&$
-4(q-q^{-1})^{\dbar}q^{\dbar} a^{2\dbar} b^{-2\dbar}
$\\
\cline{2-3}

& $i=-2$
&$
2(q-q^{-1})^{\dbar}a^{\dbar} b^{-\dbar}
\left(
a^{\dbar} b^{-\dbar}T_{\dbar}(\Lambda)
+3c^{\dbar}+c^{-\dbar}
\right)
$\\
\cline{2-3}

& $i=-1$
&$
-(q-q^{-1})^{\dbar}q^{\dbar}
\left(
a^{\dbar} b^{-\dbar}(3 c^{\dbar}+c^{-\dbar}) T_{\dbar}(\Lambda)
+
(a^{2\dbar}+3)(b^{-2\dbar}+1)+2c^{2\dbar}
\right)
$\\
\cline{2-3}

&\multirow{2}{*}{$i=0$}
&$(q-q^{-1})^{\dbar}
\big(
(b^{-2\dbar}+2+c^{2\dbar})T_{\dbar}(\Lambda)
$\\

& &$
\;+\;b^{-\dbar}c^{\dbar}
(a^{\dbar}+a^{-\dbar})
(b^{2\dbar} +2+c^{-2\dbar})
+2a^{-\dbar}b^{-\dbar}c^{\dbar}
\big)
$\\
\cline{2-3}

& $i=1$
&$
-(q-q^{-1})^{\dbar} q^{\dbar}
\left(
a^{-\dbar} c^{\dbar}(b^{\dbar}+b^{-\dbar}) T_{\dbar}(\Lambda)
+(a^{-2\dbar}+1)(c^{2\dbar}+1)
+2
\right)
$\\
\cline{2-3}

& $i=2$
&$
2 (q-q^{-1})^{\dbar} a^{-\dbar} b^{\dbar} c^{\dbar}
$\\

\midrule[1pt]

\multirow{11}{*}{$k^{\dbar i}$}

&$i=-4$
&$
-6 q^{\dbar} a^{2\dbar} b^{-2\dbar}
$\\
\cline{2-3}

&$i=-3$
&$
6a^{\dbar} b^{-\dbar}
\left(
a^{\dbar} b^{-\dbar} T_{\dbar}(\Lambda)
+c^{\dbar}+c^{-\dbar}
\right)
$\\
\cline{2-3}

&\multirow{2}{*}{$i=-2$}
&$
-q^{\dbar}
\big(
a^{2\dbar} b^{-2\dbar}T_{\dbar}(\Lambda)^2
+6 a^{\dbar} b^{-\dbar} (c^{\dbar}+c^{-\dbar})T_{\dbar}(\Lambda)
$\\

& &$
\;+\;
(a^{2\dbar}+3)(b^{-2\dbar}+3)+3a^{2\dbar}b^{-2\dbar}
+c^{2\dbar}-3+c^{-2\dbar}
\big)
$
\\
\cline{2-3}

&\multirow{2}{*}{$i=-1$}
&$
a^{\dbar} b^{-\dbar}(c^{\dbar}+c^{-\dbar})
\left(
T_{\dbar}(\Lambda)^2
+(a^{-2\dbar}+2)(b^{2\dbar}+2)
+1
\right)
$\\

& &$
\;+\;\left(
(a^{2\dbar}+2) (b^{-2\dbar}+2)+c^{2\dbar}
+3+c^{-2\dbar}
\right)
T_{\dbar}(\Lambda)
$\\

\cline{2-3}

&\multirow{2}{*}{$i=0$}
&$
-q^{\dbar}
\big(
T_{\dbar}(\Lambda)^2+
(a^{\dbar}+a^{-\dbar})
(b^{\dbar}+b^{-\dbar})
(c^{\dbar}+c^{-\dbar})T_{\dbar}(\Lambda)
$\\

& &$
\;+\;a^{2\dbar}+b^{2\dbar}+c^{2\dbar}
+8
+a^{-2\dbar}+b^{-2\dbar}+c^{-2\dbar}
\big)
$\\
\cline{2-3}

&$i=1$
&$
(a^{-2\dbar}+2+b^{2\dbar}) T_{\dbar}(\Lambda)
+
a^{-\dbar}b^{\dbar}(
a^{2\dbar}
+2
+b^{-2\dbar}
)
(c^{\dbar}+c^{-\dbar})
$\\
\cline{2-3}

&$i=2$
&$
-q^{\dbar}(a^{-2\dbar}+b^{2\dbar})
$\\

\midrule[1pt]

\multirow{7}{*}{$k^{\dbar i}e^{\dbar}$}

&$i=-4$
&$
4(q-q^{-1})^{\dbar} a^{2\dbar} b^{-2\dbar}
$\\
\cline{2-3}

&$i=-3$
&$
-2 (q-q^{-1})^{\dbar} q^{\dbar} a^{\dbar} b^{-\dbar}
\left(
a^{\dbar} b^{-\dbar} T_{\dbar}(\Lambda)+c^{\dbar}+3c^{-\dbar}
\right)
$\\
\cline{2-3}

&$i=-2$
&$
(q-q^{-1})^{\dbar}
\left(
a^{\dbar} b^{-\dbar}(c^{\dbar}+3c^{-\dbar})T_{\dbar}(\Lambda)
+(a^{2\dbar}+1)(b^{-2\dbar}+3)+2c^{-2\dbar}
\right)
$\\
\cline{2-3}

&\multirow{2}{*}{$i=-1$}
&$
-(q-q^{-1})^{\dbar} q^{\dbar}\big(
(a^{2\dbar}+2+c^{-2\dbar})T_{\dbar}(\Lambda)
$\\

& &$
\;+\;
a^{\dbar} c^{-\dbar}
(b^{\dbar}+b^{-\dbar})(a^{-2\dbar}+2+c^{2\dbar})
+2 a^{\dbar} b^{\dbar} c^{-\dbar}
\big)
$\\
\cline{2-3}

&$i=0$
&$
(q-q^{-1})^{\dbar}
\left(
b^{\dbar} c^{-\dbar}(a^{\dbar}+a^{-\dbar})T_{\dbar}(\Lambda)
+(b^{2\dbar}+1)(c^{-2\dbar}+1)+2
\right)
$\\
\cline{2-3}

&$i=1$
&$
-2(q-q^{-1})^{\dbar}
q^{\dbar} a^{-\dbar}b^{\dbar} c^{-\dbar}
$\\

\midrule[1pt]

\multirow{5}{*}{$k^{\dbar i}e^{2\dbar}$}

&$i=-4$
&$
-(q-q^{-1})^{2\dbar}q^{\dbar} a^{2\dbar} b^{-2\dbar}
$\\
\cline{2-3}

&$i=-3$
&$
2(q-q^{-1})^{2\dbar} a^{\dbar} b^{-\dbar}c^{-\dbar}
$\\
\cline{2-3}

&$i=-2$
&$
-(q-q^{-1})^{2\dbar} q^{\dbar}(a^{2\dbar}+c^{-2\dbar})
$\\
\cline{2-3}

&$i=-1$
&$
2(q-q^{-1})^{2\dbar}a^{\dbar} b^{\dbar} c^{-\dbar}
$\\
\cline{2-3}

&$i=0$
&$
-(q-q^{-1})^{2\dbar}q^{\dbar} b^{2\dbar} c^{-2\dbar}
$
\end{tabular}
\end{table}

\end{document}